\documentclass[a4paper,11pt,twoside]{article}

\usepackage{a4wide}
\usepackage{amsmath,amssymb,amsfonts,amsthm, MnSymbol}
\usepackage{moreverb,rotating,graphics}
\usepackage[hidelinks]{hyperref}
\usepackage[T1]{fontenc}
\usepackage[utf8]{inputenc}
\usepackage{epsfig}
\usepackage[francais,english]{babel} 
\usepackage{framed,fancybox}
\usepackage{tikz,pgfplots}
\usepackage{graphicx}
\usepackage{caption}
\usepackage{subcaption}
\usepackage{mathrsfs}
\usepackage{esint}
\usepackage{dsfont}
\usepackage{enumerate}
\usepackage{import}
\usepackage[normalem]{ulem}
\usepackage{color}
\usepackage{afterpage}

\numberwithin{equation}{section} % For numbering the equations with n°section.

\newtheorem{theorem}{Theorem}[section]

\newtheorem{remark}[theorem]{Remark}
\newtheorem{lemma}[theorem]{Lemma}

\newtheorem{definition}[theorem]{Definition}

\newcommand\1{{\ensuremath {\mathds 1} }} 

\def\C{{\mathbb C}}

\def\N{{\mathbb N}}

\def\R{{\mathbb R}}

\def\Z{{\mathbb Z}}

\def\bk{{\bold k}}

\def\bl{{\bold \ell}}

\def\bR{{\bold R}}

\def\bX{{\bold X}}
\def\bx{{\bold x}}
\def\by{{\bold y}}

\def\bnull{{\bold 0}}

\def\rd{{\mathrm{d}}}
\def\re{{\mathrm{e}}}
\def\ri{{\mathrm{i}}}

\def\cA{{\mathcal A}}
\def\cB{{\mathcal B}}

\def\cE{{\mathcal E}}

\def\cH{{\mathcal H}}

\def\cK{{\mathcal K}}

\def\cP{{\mathcal P}}

\def\cS{{\mathcal S}}

\newcommand{\bra}{\langle}
\newcommand{\ket}{\rangle}

\newcommand{\p}{\partial}

\newcommand{\WS}{\Gamma} %cellule Wigner-Seitz = \R^3 / \Lat
\newcommand{\RLat}{\mathcal{R}^*} % Reciprocal Lattice
\newcommand{\Lat}{\mathcal{R}} % Lattice
\newcommand{\loc}{\mathrm{loc}}

\newcommand{\rin}{\mathrm{in}}

\newcommand{\Ker}{\mathrm{Ker}}

\newcommand{\Span}{\mathrm{Span}}

\def\sqw{\hbox{\rlap{\leavevmode\raise.3ex\hbox{$\sqcap$}}$%
\sqcup$}}
\def\cqfd{\ifmmode\sqw\else{\ifhmode\unskip\fi\nobreak\hfil
\penalty50\hskip1em\null\nobreak\hfil\sqw
\parfillskip=0pt\finalhyphendemerits=0\endgraf}\fi}

\hypersetup{
    colorlinks,
    linkcolor={blue!50!blue},
    citecolor={red}
}
\renewcommand{\eqref}[1]{(\ref{#1})}

\begin{document}

\title{A mathematical and numerical framework for bubble meta-screens\thanks{\footnotesize Hyundae Lee was supported by NRF-2015R1D1A1A01059357 grant.  Hai Zhang was supported by the initiation grant IGN15SC05 from HKUST.}}
\date{}

\author{
Habib Ammari\thanks{\footnotesize Department of Mathematics, 
ETH Z\"urich, 
R\"amistrasse 101, CH-8092 Z\"urich, Switzerland (habib.ammari@math.ethz.ch, brian.fitzpatrick@sam.math.ethz.ch, david.gontier@sam.math.ethz.ch ).} \and Brian Fitzpatrick\footnotemark[2] \and David Gontier\footnotemark[2] 
\and Hyundae Lee\thanks{\footnotesize  Department of Mathematics, Inha University,  253 Yonghyun-dong Nam-gu,  Incheon 402-751,  Korea (hdlee@inha.ac.kr).}  \and Hai Zhang\thanks{\footnotesize 
Department of Mathematics, 
 HKUST,  Clear Water Bay, Kowloon, Hong Kong (haizhang@ust.hk).}}

\maketitle

\begin{abstract}
The aim of this paper is to provide a mathematical and numerical framework for the analysis and design of bubble meta-screens. An acoustic meta-screen is a thin sheet with patterned subwavelength structures, which nevertheless has a macroscopic effect on the acoustic wave propagation. In this paper, periodic subwavelength bubbles mounted on a reflective surface (with Dirichlet boundary condition) is considered. It is shown that the structure behaves as an equivalent surface with Neumann boundary condition at the Minnaert resonant frequency which corresponds to a wavelength much greater than the size of the bubbles. Analytical formula for this resonance is derived. Numerical simulations confirm its accuracy and show how it depends on the ratio between the periodicity of the lattice, the size of the bubble, and the distance from the reflective surface. The results of this paper formally explain the super-absorption behavior observed in~[V.~Leroy et al., Phys. Rev. B, 2015].

\end{abstract}

\bigskip

\noindent {\footnotesize Mathematics Subject Classification
(MSC2000): 35R30, 35C20.}


%\noindent \textcolor{red}{{\footnotesize Short title: Bubble metasurfaces}.}

%%%%%%%%%%%%%%%%%%%%%%%%%%%%%%%%%%%%%%%%%%

\section{Introduction}

In this article we study the reflection properties of a {\em meta-screen} from a mathematical point of view. Broadly speaking, a meta-screen is a thin sheet with  patterned subwavelength structures which has a macroscopic effect on the reflection and transmission of waves. 

\medskip

One way to design meta-screens is to put microscopic gas inclusions along a periodic lattice. The properties of such screens have been studied in recent years, with spectacular results. In~\cite{Leroy2009, Leroy2009_a}, it was experimentally shown how the reflection and transmission coefficients vary with respect to the wavelength of the incoming acoustic wave. It was later shown how super-absorption may be achieved with these meta-screens~\cite{Leroy2015}, \textit{i.e.} null reflection and null transmission coefficients.

\medskip

These phenomena can be explained by the use of subwavelength resonators in the design of the meta-screens. In~\cite{prsa} for instance, a mathematical justification was given when the meta-screen was made of subwavelength plasmonic particles, where resonance is due to negative dielectric coefficient~\cite{Ammari2015_c}. Minnaert bubbles act like plasmonic nanoparticles. We refer the reader to \cite{Ammari2015_c,Ammari2016, kang1, hyeonbae, Gri12} for the mathematical analysis of resonances for plasmonic nanoparticles. 

\medskip

In this paper, we study the case where the resonance is due to the high contrast in density between the gas inclusions and the surrounding medium, characterized by a small parameter $\delta$ representing the inverse of the contrast. This resonance, known as the Minnaert resonance, was observed and explained as early as 1933~\cite{Minnaert1933} (see also~\cite{Leroy2002, Devaud2008}). We recently gave a rigorous mathematical justification of this resonance in the case of a single bubble in a homogeneous medium~\cite{Ammari2016_Minnaert}. We studied the acoustic property of the bubble, and proved that Minnaert resonance occurred when the frequency is 
appropriately propositional to $\frac{\sqrt{\delta}}{s}$ with $s$ being the size of the bubble.

\medskip

In the present paper, we use similar techniques as in~\cite{prsa} to study the reflection properties of a meta-screen when both the typical size of the periodic cell and the size of the bubbles are subwavelength, and the contrast $\delta^{-1}$ is large. We also investigate the limiting case when $a$ goes to $0$ and $\delta$ goes to $0$ proportional to $a^2$. Note that in the case where $\delta$ is a fixed parameter, the usual homogenization techniques can be applied, and a large literature already exists on describing boundary layer effects~\cite{Abboud1996, Achdou1998, Allaire1999}. In our case, because of the excitation of resonace, one must resort to other methods. The idea of using an unbounded parameter as the size of the cells goes to zero is not new, and is sometimes referred to as <<high-contrast homogenization>>~\cite{Briane2012, Camar2013, Felbacq1997}. This technique has already been successful in explaining some spectacular physical phenomena~\cite{Bouchitte2004}.

\medskip

Our main result is the following. We consider a periodic subwavelength bubbles above a reflective surface ( with Dirichlet boundary condition). Then, under the appropriate scaling $\delta = \mu a^2$, Minnaert resonance can be excited at a fixed frequency of order one, and 
the surface behaves as an equivalent surface with Neumann boundary condition at this frequency at the limit $a \to 0$. The theorem is valid for all shapes of Minnaert resonators. When taking into account some extra physical damping effects that do not appear in our mathematical model~\cite{Khismatullin2004}, this eventually explains the super-absorption behavior witnessed in~\cite{Leroy2015} (see Remark \ref{remark23}).

\medskip

The paper is structured as follows. In Section~\ref{sec:statementPb}, we fix the notations of the experiment under consideration, and state our main result, the proof of which is detailed in Section~\ref{sec:proof}. The proof uses layer potential techniques and asymptotic expansions. Finally, numerical results are presented in Section~\ref{sec:numericalIllustrations}.

%%%%%%%%%%%%%%%%%%%%%%%%%%%%%%%%%%
\section{Statement of the problem}
\label{sec:statementPb}

We consider a reflective surface, on top of which small scatterers are arranged along some periodic lattice. A typical example of such a situation is given by air bubbles arranged in water, as described in~\cite{Leroy2015, Leroy2009_a}. 

\medskip

Let us fix some notation. We will state our results in dimension $d \in \{ 2, 3 \}$, and write $\R^d \ni \bx = (\bar{x}, x_d)$, with $\bar{x} \in \R^{d-1}$ and $x_d \in \R$. We let $\partial \R^d_+ := \{ \bx \in \R^d, \ x_d = 0\}$ represent the reflective plane, and $\R^d_{\pm} := \{ \bx \in \R^d, \ \pm x_d > 0)$ be the upper and lower half space. The shape of the bubbles is described by a simply connected domain $D \subset \R^d_+$ with smooth boundary $\p D$. The bubbles are arranged periodically along a lattice $\Lat$ of $\R^{d-1}$. For instance, if $d = 2$, then $\Lat = a \Z$ for some $a > 0$. 

\medskip

For $\varepsilon > 0$, we denote by $\Omega^\varepsilon$ the volume occupied by the bubbles. More specifically, we set
\[
	\Omega^\varepsilon := \bigcup_{\bR \in \Lat} \varepsilon \left( D + \bR \right)
\]

\begin{figure}[h]
\centering
\begin{tikzpicture}
	%\draw (0,0) rectangle (10, 5);
	\draw[->] (-1, 1) -> (0, 1); \node at (0, 1.25) {$\bar{x}$};
	\draw[->] (-1, 1) -> (-1, 2); \node at (-0.7, 1.7) {$x_d$};

	\draw (1,1) -- (9,1);
	\foreach \x in {1, 2, 3, 4, 5, 6, 7, 8} {
		\draw (\x, 0.5) -- (\x+0.5, 1);
		\draw (\x+0.5, 0.5) -- (\x+1, 1);
	}
	\node at (9.5,1) {$\p \R^d_+$};
	
	\foreach \x in {2, 3, 4, 5, 6, 7, 8} {
		\draw (\x, 2) circle (0.3);
	}
	
	\draw[<->] (3, 2.5) -- (4,2.5); \node at (3.5, 2.7) {$\varepsilon a$};
	\draw[<->] (1.3, 1) -- (1.3,2); \node at (0.8, 1.5) {$O(\varepsilon)$};
	
\end{tikzpicture}
\end{figure}
%
%Throughout the paper, macroscopic variables are denoted with capital letters $\bX, \bY,...$, while their respective microscopic variables are denoted with lower case letters $\bx = \bX/\varepsilon$, $\by = \bY/\varepsilon$, ....
\medskip

We denote by $\rho_b$ and $\kappa_b$ the density and bulk modulus of the air inside the bubbles, and by $\rho$ and $\kappa$ the corresponding parameters for the background medium $\R^d \setminus \overline{\Omega^\varepsilon}$. We consider the scattering of acoustic waves by this meta-screen. In the sequel, $\omega/(2\pi)$ represents the frequency of the source,  and
\[
	v = \sqrt{\frac{\lambda}{\rho}}, \quad 
	v_b = \sqrt{\frac{\lambda_b}{\rho_b}}, \quad
	k= \frac{\omega}{v}, \quad \text{and} \quad
	k_b= \frac{\omega}{v_b}
\]
denote respectively the speed of sound outside and inside the bubbles, and the wave number outside and inside the bubbles. Finally, we introduce the dimensionless contrast parameter
\[
	\delta = \frac{\rho_b}{\rho}.
\]

We assume that $\delta \ll 1$. 
With appropriate physical units, we also assume that $k=O(1)$, $k_b=O(1)$, and
the size of $D$ is also of order one. 
With these in mind, the acoustic problem is (we use capital letters for macroscopic fields)
\begin{equation} \label{eq:scattering}
	\left\{ \begin{aligned}
		& \left( \Delta + k^2  \right) U^\varepsilon = 0 \quad \text{on} \quad \R^d_+ \setminus \overline{\Omega^\varepsilon}, \\
		& \left( \Delta + k_b^2  \right) U^\varepsilon = 0 \quad \text{on} \quad \Omega^\varepsilon ,\\
		& U^\varepsilon |_+ = U^\varepsilon |_- \quad \text{on} \quad \p \Omega^\varepsilon ,\\
		& \partial_\nu U^\varepsilon |_- = \delta \partial_\nu U^\varepsilon|_+  \quad \text{on} \quad \p \Omega^\varepsilon , \\
		& U^s := U^\varepsilon - U^\rin \quad \text{satisfies the outgoing radiation condition} , \\
		& U^\varepsilon = 0 \quad \text{on} \quad \p \R^d_+,
	\end{aligned}
	\right.
\end{equation}
where $U^\rin$ is some incoming pressure wave satisfying $(\Delta + k^2) U^\rin = 0$ and $\cdot|_\pm$ denotes the limits from respectively outside and inside of $\Omega^\varepsilon$.
%The outgoing radiation condition is quite complicated to write in the general case. 
In this paper, we consider the special case where the incoming pressure $U^\rin$ is a plane wave going towards the plane from the upper half-space, so we set
\[
	U_\bk^\rin(\bX) = u_0 \re^{- \ri \bk \cdot \bX} = u_0 \re^{- \ri \bar{k} \cdot \bar{X}} \re^{- \ri k_d X_d},
\]
where $\bk = (\bar{k}, k_d) \in \R^{d-1} \times \R$, with $k_d > 0$ and $| \bk | = k$, is the wave vector. 
%
%In this case, the outgoing radiation condition takes the form
%\begin{equation} \label{eq:1dRadiation}
%	U^s(\bX) := U(\bX) - U^\rin_\bk(\bX) \sim \re^{- \ri \bar{k} \cdot \bar{X}} \re^{\ri k_d X_d} \quad \text{as} \quad X_d \to \infty.
%\end{equation}
%In the general case where $U^\rin$ is a superposition of planewaves, we can decompose $U^\rin$ using Bloch-Flocquet theory~\cite{ReedSimon4, Ammari2009_book}. We obtain a family of problems to solve, each one with its own outgoing radiation condition. The final solution is then the superposition of all these solutions. We do not comment further on this point, as it is quite standard.

\medskip

Such problems have been extensively studied using homogenization theory in recent decades. For instance, the scattered field is well-understood as $\varepsilon \to 0$, and is of order $\varepsilon$~\cite{Achdou1998, Allaire1999, Ammari2016_plasmonicMetasurfaces}. In the present article, we study the special case where the contrast $\delta$ is also scaled with $\varepsilon \to 0$. As we will see, such a regime, which is accessible physically, presents some interesting features.

\medskip

More specifically, according to~\cite{Ammari2016_Minnaert}, there is a resonance phenomenon in the regime $\sqrt{\delta} \sim \varepsilon$ (Minnaert resonance). In the sequel, we fix $\mu > 0$, and study~\eqref{eq:scattering}, with 
\[
	\delta := \delta_\varepsilon = \mu \varepsilon^2.
\]
In this case, standard homogenization techniques are no longer applicable and new techniques are needed.

\medskip

In the absence of bubbles, the solution of~\eqref{eq:scattering} is simply
\[
	U_0(\bX) := U_\bk^\rin(\bar{X}, X_d) - U_\bk^\rin(\bar{X}, -X_d) = - 2 \ri u_0 \re^{- \ri \bar{k} \cdot \bar{X}} \sin(k_d X_d).
\]
In the presence of bubbles, we expect this solution to be perturbed. Our goal is to describe the main contribution of this perturbation. In order to state our results, we introduce some extra-notation. First, we introduce a constant $\mu_M$ defined by
\begin{equation} \label{eq:def:muM}
	\mu_M := \dfrac{k_b^2 | D | }{C_{D, \Lat}^+},
\end{equation}
where $| D |$ is the volume of $D$, and $C_{D, \Lat}^+$ is the periodic capacity defined in Definition~\ref{def:periodicCapacity}. Then we introduce the scattering function
\begin{equation} \label{eq:def:gs}
	g_s(\mu, \varepsilon) := 
	 \dfrac{ \varepsilon  M_1}{1 - \frac{\mu_M}{\mu}  - \varepsilon \ri \frac{M_1^2 k_d C_{D, \Lat}^+ }{| \WS |}},
\end{equation}
where the constant $M_1$ is defined in~\eqref{eq:def:M1} below and $|\Gamma|$ is the volume of the fundamental cell of $\Lat$. We also introduce the functions $\alpha_0$ and $\alpha_1$: $\alpha_0$ is the (unique) solution to the problem
\begin{equation} \label{eq:def:alpha0}
	\left\{ \begin{aligned}
		& \Delta \alpha_0 = 0 \quad \text{on} \quad \R^d_+ \setminus \overline{\Omega^1}, \\
		& \alpha_0 |_+ = 1 \quad \text{on} \quad \p \Omega^1, \\
		& \alpha_0 - \alpha_{0, \infty} \quad \text{is exponentially decaying as $x_d \to \infty$}, \\
		& \alpha_0 = 0 \quad \text{on} \quad \p \R^d_+,
	\end{aligned}
	\right.
\end{equation}
and $\alpha_1$ is the (unique) solution to the problem
\begin{equation} \label{eq:def:alpha1}
	\left\{ \begin{aligned}
		& \Delta \alpha_1 = 0 \quad \text{on} \quad \R^d_+ \setminus \overline{\Omega^1}, \\
		& \alpha_1 |_+ = x_d |_+ \quad \text{on} \quad \p \Omega^1, \\
		& \alpha_1 - \alpha_{1, \infty} \quad \text{is exponentially decaying as $x_d \to \infty$}, \\
		& \alpha_1 = 0 \quad \text{on} \quad \p \R^d_+.
	\end{aligned}
	\right.
\end{equation}
The exact values of $\alpha_{0,\infty}$ and $\alpha_{1,\infty}$ are given in Lemma~\ref{lem:alpha01} below. The function $\alpha_0$ is related to the monopole moment of bubbles, while $\alpha_1$ is related to their dipole moment.

\medskip

Finally, we introduce a functional space. Let $L \in \R^+$ be large enough so that for all $\by \in D$, it holds that $y_d \le L$ (hence $0 \le y_d \le L$). For $a \in \R^+$, we denote by $S_a := \R^{d-1} \times (a, \infty)$, and we denote by $W^{1, \infty}(S_a)$ the usual Sobolev space with norm
\[
	\left\| f \right\|_{W^{1, \infty}(S_a)} := \sup_{\bX \in S_a} \left| f \right| (\bX)+   \sup_{\bX \in S_a} \left| \nabla f \right|( \bX).
\]

\medskip

Our main result is the following.
\begin{theorem} \label{th:resonant_metascreen}
There exists $C \in \R^+$ such that, for all $\varepsilon \ge 0$ and all $\mu > 0$, it holds that
	\[
		\left\| U^\varepsilon(\bX) - U_0(\bX) - \left( U_1^\varepsilon(\bX) + U_{\rm BL}^\varepsilon \left( \bX, \frac{\bX}{\varepsilon}\right) \right) \right\|_{W^{1, \infty}(S_{\varepsilon L})} \le C \varepsilon \left| g_s(\mu, \varepsilon) \right|,
	\]
	where 
	\begin{eqnarray*}
		U_{\rm BL}^\varepsilon(\bX, \bx) &=& (2 \ri u_0 k_d)  \re^{- \ri \bar{k} \cdot \bar{X}}\left[ \varepsilon \left( \alpha_1(\bx) - \alpha_{1, \infty} \right) - g_s(\mu, \varepsilon)  \left( \alpha_0(\bx) - \alpha_{0, \infty} \right) \right], \\
		U_1^\varepsilon(\bX) &=& (2 \ri u_0 k_d)  \re^{\ri (k_d X_d - \bar{k} \cdot \bar{X})} \left( \varepsilon \alpha_{1, \infty} - g_s(\mu, \varepsilon) \alpha_{0, \infty} \right),
	\end{eqnarray*}
	and $g_s(\mu, \varepsilon)$ is defined by (\ref{eq:def:gs}). Moreover, 
	$U_{\rm BL}^\varepsilon$ is exponentially decaying as $x_d \to \infty$.\end{theorem}

The proof of the above theorem will be given in Section \ref{sec:proof}. 
The function $U_{\rm BL}^\varepsilon$ describes the behavior of the field near the bubbles (at distances of order $O(\varepsilon)$), while the function $U_1^\varepsilon$ describes the far field. Both functions have an $O(\varepsilon)$ part due to the dipole moment, and a resonant part $g_s(\mu, \varepsilon)$ due to the monopole moment.  

\medskip

	 In particular, from the behavior of $U^\varepsilon$ in the far-field, namely
	\[
		U^\varepsilon (\bX) \approx u_0 \re^{- \ri (k_d X_d + \bar{k} \cdot \bar{X})} + u_0  \left( - 1 +  2 \ri k_d \left( \varepsilon \alpha_{1, \infty} - g_s(\mu, \varepsilon) \alpha_{0, \infty} \right) \right)  \re^{\ri (k_d  X_d - \bar{k} \cdot \bar{X})},
	\]
	we directly read the reflection coefficient
	\[
		R_\varepsilon = - 1 +  2 \ri k_d  \left( \varepsilon \alpha_{1, \infty} - g_s(\mu, \varepsilon) \alpha_{0, \infty} \right).
	\]

As a consequence, we obtain
\begin{theorem}	
%	For instance, if $\mu \neq \mu_M$, then in the limit $\varepsilon \to 0$, the bubbles act as a small perturbation of the Dirichlet case. 
In the case when $\mu = \mu_M$, in the limit as $\varepsilon \to 0$, we get
	\[
		R_\varepsilon = -1 + 2 = 1.
	\]
Therefore the equivalent screen has Neumann boundary condition for the wave equation.
\end{theorem}

\begin{remark}
	If we neglect the effect of $\alpha_{1, \infty}$ for simplicity, then 
		\[
		R_\varepsilon \approx -1 - 2 \left( \frac{\varepsilon \ri k_d M_1^2 C_{D, \Lat}^+ | \WS |^{-1}}{1 - \frac{\mu_M}{\mu} - \ri \varepsilon \frac{k_d M_1^2 C_{D, \Lat}^+}{| \WS |}} \right).
	\]
Using the frequency variable $\omega$, we see that 
	\begin{equation} \label{eq:Repsilon}
		R(\omega) \approx -1 - 2  \left( \frac{\ri \frac{\omega M_1^2 C_{D, \Lat}^+}{ v_d | \WS |}}{1 - \left(\frac{\omega}{\omega_M^+} \right)^2 - \ri  \frac{\omega M_1^2 C_{D, \Lat}^+}{ v_d | \WS |}} \right),
		\end{equation}
where $\omega_M^+ := \sqrt{\frac{\delta C_{D, \Lat}^+}{| D |}} v_b$	is the (periodic) Minnaert resonant frequency. Note that it is similar in expression to the usual Minnaert resonance $\omega_M$~\cite{Ammari2016_Minnaert, Minnaert1933}. Actually, if the bubbles are small compared to the typical size of the lattice, and are far away from each other, then $\omega_M^+ \approx \omega_M$.
\end{remark}

\begin{remark} \label{remark23}
	The term $\eta_{\rm rad} := \frac{\omega M_1^2 C_{D, \Lat}^+}{ v_d | \WS |}$ in the denominator of~\eqref{eq:Repsilon} is called the radiative damping. It is possible to include more realistic damping effects~\cite{Khismatullin2004}. In this case, one should replace $\eta_{\rm rad}$ by $\eta := \eta_{\rm rad} + \eta_{\rm other}$, where $\eta_{\rm other}$ includes all the remaining sources of damping. Note that in the particular case where $\eta_{\rm other} = \eta_{\rm rad}$ so that $\eta = 2 \eta_{\rm rad}$, then at the resonant frequency $\omega = \omega_M$ we obtain that $R(\omega) \approx 0$. This eventually explains the super-absorption phenomenon experimented in~\cite{Leroy2015}: all the incoming energy is dissipated with damping effects.
\end{remark}

The proof of Theorem~\ref{th:resonant_metascreen} is given in the next section. It relies on the theory of periodic layer potentials. 

%%%%%%%%%%%%%%%%%%%%%%%%%%%%%%%%%%

\section{Proof of Theorem~\ref{th:resonant_metascreen}.}
\label{sec:proof}

%%%%%%%%%%%%%%%%%%%%%%%%%%%%%%%%%%%%%%%%%%%%%%%%%%%

%Note that we have $\widetilde{U^{\rin}_{\bk}} (\bX) = \widetilde{U^\rin_{\varepsilon \bk}}(\bx)$. 
%With this change of variable,~\eqref{eq:scattering} is equivalent to solving
%\begin{equation} \label{eq:scatteringPb2}
%	\left\{ \begin{aligned}
%		& \left( \left( \nabla_{\bar{x}} - \ri \varepsilon \bar{k}\right)^2 + \partial^2_{x_d x_d} + \varepsilon^2 k^2 \right) u^\varepsilon = 0 \quad \text{on} \quad \R^d \setminus \overline{\Omega^1} \\
%		& \left( \left( \nabla_{\bar{x}} - \ri \varepsilon \bar{k}\right)^2 + \partial^2_{x_d x_d} + \varepsilon^2 k_b^2 \right) u^\varepsilon = 0 \quad \text{on} \quad \Omega^1\\
%		& u^\varepsilon |+ = u^\varepsilon |_- \quad \text{on} \quad \p \Omega^1 \\
%		& \partial_\nu u^\varepsilon |_- = \mu \varepsilon^2  \partial_\nu u^\varepsilon|_+  \quad \text{on} \quad \p \Omega^1 \\
%		& u^\varepsilon - \widetilde{U^\rin_{\varepsilon \bk}} \quad \text{satisfies the outgoing radiation condition}\\
%		& u^\varepsilon = 0 \quad \text{on} \quad \p \R^2_+\\
%		& u^\varepsilon \quad \text{is $\Lat$-periodic}.
%		\end{aligned}
%	\right.
%\end{equation}
%

%%%%%%%%%%%%%%%%%%%%%%%%%%%%%%%%%%
\subsection{Periodic Green's functions}

We recall in this section the definition and properties of periodic layer potentials~\cite[Part 3]{Ammari2009_book}.
Recall that $\Lat$ is a lattice of the plane $\R^{d-1}$. We let $\RLat$ be its reciprocal lattice, and denote by $\WS$ the unit cell of $\Lat$. For instance, if $d = 2$ with $\Lat = a \Z$, then $\RLat = (2 \pi/a) \Z$ and $\WS = (-a/2, a/2)$.

%%%%%%%%%%%%%%%%%
\paragraph{Periodic Green's function without Dirichlet boundary condition.}

%%%%%%%%%%%%%%%%%%%%%%%%%%%%%%%%%%

We first introduce, for $\bk \in \R^d$, the periodic Green function $G_\sharp^\bk$, solution to
\begin{equation} \label{eq:Gksharp}
	 \left( \Delta + k^2  \right)  G_\sharp^\bk(\bx) := \sum_{\bR \in \Lat} \re^{-\ri \bar{k} \cdot \bar{x}} \delta_{\bR}(\bx) = \sum_{\bR \in \Lat} \re^{-\ri \bar{k} \cdot \bR} \delta_{\bR}(\bx),
\end{equation}
with the outgoing radiation condition. Here and after $\delta_{\bR}(\bx)$ denotes the Dirac mass at the point $\bR$.  
It holds that $G_\sharp^\bk(\bx, \by) = G_\sharp^k (\bx - \by, 0) :=  G_\sharp^k (\bx - \by)$. 
%The exact expression can be calculated (see Appendix~\ref{sec:app:periodicGF} for details).

The following lemma is needed.

\begin{lemma} \label{lemadd}
		The solutions to $(\partial_{xx}^2 + \alpha) f_\alpha = \delta_0$ are
		\[
			f_\alpha(x) = \left\{ \begin{aligned}
				& \frac12 | x |,  \quad \quad \quad \alpha = 0, \\
				& \dfrac{ 1}{2 \ri \sqrt{\alpha}} \re^{\ri \sqrt{\alpha} | x |},  \quad \alpha>0,\\
				& \dfrac{ 1}{2 \sqrt{-\alpha}} \re^{ \sqrt{-\alpha} | x |},  \quad \alpha<0.
			\end{aligned}
			\right.
		\]
	\end{lemma}
\begin{proof}
	It is enough to check that $f_\alpha$ is the solution to $(\partial_{xx}^2 + \alpha) f_\alpha = \delta_0$. Recall that in the sense of distributions, $| x |' = 2 \Theta(x) - 1$, where $\Theta$ is the Heaviside function $\Theta (x) := \1(x > 0)$, and that $| x|'' = 2 \Theta'(x) = 2 \delta_0$. Note also that $\left( |x|' \right)^2 = 1$. The proof follows by standard calculations.
\end{proof}
The following result holds. 
\begin{lemma} \label{lem:Gk} If $\bk = \bnull$, then
\begin{equation} \label{eq:G0expansion}
	G_\sharp^\bnull((\bar{x}, x_d)) = \frac{ | x_d |}{2 | \WS |} - \sum_{\bl \in \RLat \setminus \{ \bnull \}} \frac{1}{2  | \WS | \, | \bl |} \re^{- | \bl | \, |x_d|} \re^{ \ri (\bl \cdot \bar{x})}. 
\end{equation}
If $\bk \in \R^d$ satisfies $k^2 < \inf \{ | \bl - \bar{k}|, \ \bl \in \RLat \setminus \{ \bnull \} \}$, then
\begin{equation} \label{eq:expansionGksharp}
	G_\sharp^\bk((\bar{x}, x_d)) = \dfrac{\re^{-\ri \bar{k} \cdot \bar{x}} \re^{\ri k_d  \, |x_d|}}{2 \ri k_d | \Gamma|} - \sum_{\bl \in \RLat \setminus\{ 0 \}} \frac{\re^{-\ri \bar{k} \cdot \bar{x}}}{2  | \WS | \sqrt{| \bl - \bar{k}|^2 - k^2}} \re^{- \sqrt{ | \bl - \bar{k}|^2 - k^2} | x_d |} \re^{\ri (\bl \cdot \bar{x})}.
\end{equation}
\end{lemma}

%%%%%%%%%%%%%%%%%%%%%%%%%%%%%%%%%%

\begin{proof}
In order to compute $G_\sharp^\bk(\bx)$, we introduce $\widetilde{G_\sharp^\bk}(\bx) = \re^{\ri \bar{k} \cdot \bar{\bx}} G_\sharp^\bk(\bx)$, so that
$G_\sharp^\bk(\bx) = \re^{-\ri \bar{k} \cdot \bar{\bx}} \widetilde{G_\sharp^\bk}(\bx)$.
In particular, $\widetilde{G_\sharp^\bk}$ is the $\Lat$-periodic solution to $\left( \Delta - 2 \ri \bar{k} \cdot \nabla_{\bar{x}} + k_d^2  \right)  \widetilde{G_\sharp^\bk}(\bx) := \sum_{\bR \in \Lat} \delta_{\bR}(\bx)$. We consider its Fourier expansion, and write
	\[
		\widetilde{G_\sharp^\bk}(\bar{x}, x_d)  = \sum_{\bl \in \RLat} c_\bl(x_d) \re^{ \ri \bl \cdot  \bar{x}}.
	\]
	Thanks to the Poisson summation formula
	\[
		\sum_{\bR \in \Lat} \delta (\bx + \bR) = \frac{1}{| \WS |} \sum_{\bl \in \RLat} \re^{\ri \bl \cdot \bar{x}},
	\]
	we obtain that $c_\bl$ must be the solution to
	\[
		\forall \bl \in \RLat, \quad \left( \p^2_{x_d x_d} + k^2 - | \bl - \bar{k}|^2 \right) c_\bl = \frac{1}{| \WS |} \delta_0.
	\]
	The proof then follows from Lemma \ref{lemadd}. 
\end{proof}
%%%%%%%%%%%%%%%%%
\begin{remark}
It is possible to simplify~\eqref{eq:G0expansion} following~\cite[pp. 813-814]{Collin1960} (see for instance~\cite{Ammari2016_plasmonicMetasurfaces}).
\end{remark}
%%%%%%%%%%%%%%%%%

From~\eqref{eq:expansionGksharp}, we see that $G_{\sharp}^{\varepsilon \bk}$ has a formal expansion of the form
\begin{equation} \label{eq:TaylorGsharp}
	G_\sharp^{\varepsilon \bk}(\bx) = \dfrac{1}{2 \ri \varepsilon k_d | \WS |} + G_{0,\sharp}^\bk(\bx)  + \sum_{n=1}^\infty \varepsilon^n G_{n,\sharp}^\bk(\bx),
\end{equation}
where the operators $G_{n, \sharp}^\bk$ can be computed explicitly. For instance, together with 
(\ref{eq:G0expansion}), we have
\begin{equation} \label{eq:def:G0+k}
	G_{0, \sharp}^{\bk} = \dfrac{k_d | x_d | - \bar{k} \cdot \bar{x}}{2 k_d | \WS |} - \sum_{\bl \in \RLat \setminus \{ \bnull \}} \frac{\re^{- | \bl | \, |x_d|} \re^{ \ri (\bl \cdot \bar{x})}}{2  | \WS | \, | \bl |} = G_\sharp^\bnull - \dfrac{\bar{k} \cdot \bar{x}}{2 k_d | \Gamma|}.
\end{equation}
We will also need the exact formula for $G_{1, \sharp}^\bk$. After some straightforward calculations we find that
\begin{equation} \label{eq:def:G1sharp}
	G_{1,\sharp}^\bk(\bx) = \frac{ \ri \left( k_d | x_d | - \bar{k} \cdot \bar{x}\right)^2}{4 k_d | \WS |}
		- \bar{k} \cdot g_{1,\sharp}(\bx),
\end{equation}
where $g_{1, \sharp}$ is a function independent of $\bk$. Explicitly,
\begin{equation*}
		g_{1,\sharp}(\bx) := \sum_{\bl \in \RLat \setminus \{ \bnull \}}  \frac{\re^{- | \bl | \, |x_d|} \re^{ \ri (\bl \cdot \bar{x})}}{2  | \WS | \, | \bl |} \left(  \frac{\bl}{| \bl |^2} + \frac{\bl}{| \bl |} | x_d | - \ri \bar{x}  \right).
\end{equation*}
By change of variable $\bl \to - \bl$, we obtain that
\begin{equation} \label{eq:g1sharp}
	g_{1,\sharp}(\bx) = \frac{\ri}{2}  \sum_{\bl \in \RLat \setminus \{ \bnull \}}   \frac{\re^{- | \bl | \, |x_d|} }{ | \WS | \, | \bl |} \left(  \left(\frac{\bl}{| \bl |^2} + \frac{\bl}{| \bl |} | x_d | \right)  \sin(\bl \cdot \bar{x}) -  \bar{x} \cos(\bl \cdot \bar{x})  \right).
\end{equation}
In particular, we see that $g_{1, \sharp}$ satisfies the symmetry relations
\begin{equation} \label{eq:relations_g1sharp}
	g_{1,\sharp}(\bar{x}, x_d) = g_{1,\sharp}(\bar{x}, -x_d)  = - g_{1,\sharp}(- \bar{x}, x_d).
\end{equation}

\medskip

Note that there is a singularity in~\eqref{eq:TaylorGsharp} as $\varepsilon$ goes to $0$.
% so that $G_\sharp^{\varepsilon k}$ does not converge to $G_\sharp^0$ as $\varepsilon$ goes to $\bnull$. 
Finally, expanding \eqref{eq:Gksharp} in powers of $\varepsilon$ leads to the equations %(written here in the variable $\by$)
\begin{equation} \label{eq:magicRelations}
\left\{ \begin{aligned}
 \Delta G_{0,\sharp}^\bk &= \sum_{\bR \in \Lat} \delta_{\bR}, \\
 \Delta G_{1,\sharp}^\bk &= \sum_{\bR \in \Lat} -\ri \bR \delta_{\bR}, \\
\Delta G_{n+2,\sharp}^\bk &+ k^2 G_{n, \sharp}^\bk = \sum_{\bR \in \Lat} \dfrac{(-\ri \bar{k} \cdot \bR )^{n+2}}{(n+2)!} \delta_{\bR}, \quad \forall n \geq 2.
\end{aligned}
\right.
\end{equation}

The periodic Green function would be an adequate tool to study the problem {\em without} the Dirichlet boundary condition $U^\varepsilon = 0$ on $\p \R^d_+$. 
%Unfortunately, we were not able to do the complete analysis in this case. 
%The  main reason being that the function $G_{1, \sharp}^k$ defined in~\eqref{eq:def:G1sharp} is quadratic in $x_d$. 
%In order to state precise statements, we study the problem {\em with} the Dirichlet boundary condition.
In this article however, we study the problem {\em with} the Dirichlet boundary condition to explain the phenomenon seen in~\cite{Leroy2015} for instance.

%%%%%%%%%%%%%%%%%%%%%%%%%%%%%%%%%%
\paragraph{Periodic Green's function with Dirichlet boundary condition}
We introduce the Dirichlet Green function defined by
\begin{equation} \label{eq:def:Gk+}
	G_{+}^\bk(\bx, \by) := G_\sharp^\bk ( ( \bar{x} - \bar{y}, x_d - y_d) ) - G_\sharp^\bk ( ( \bar{x} - \bar{y}, x_d + y_d) ).
\end{equation}
This Green function is no longer translational invariant (in the sense $G_{+}^\bk(\bx, \by) \neq G_{+}^\bk(\bx - \by, \bnull)$), but satisfies the Dirichlet boundary condition $G^\bk_+ = 0$ on $\p \R^d_+$. From~\eqref{eq:expansionGksharp}, we deduce that $G_{+}^{\varepsilon \bk}$ admits an expansion of the form
\begin{equation} \label{eq:expansionGk+}
	G_{+}^{\varepsilon \bk}(\bx, \by) = \sum_{n=0}^\infty \varepsilon^n G_{n, +}^\bk(\bx, \by),
\end{equation}
where
\[
	G_{n,+}^\bk(\bx, \by) := G_{n, \sharp}^\bk ( (\bar{x} - \bar{y}, x_d - y_d ) ) - G_{n, \sharp}^\bk ( ( \bar{x} - \bar{y}, x_d + y_d) ).
\]
Note that $G_+^{\varepsilon \bk}$ is no longer singularity as $\varepsilon$ goes to $0$. This makes the problem with Dirichlet boundary condition easier to study analytically. Moreover, from (\ref{eq:def:G0+k}) we can check that
\[
	G_{\bnull, +}^\bk(\bx, \by) = G_{+}^{\bk = \bnull}(\bx, \by).
\]
Note that this equality does not hold for the periodic Green function without Dirichlet boundary condition (see (\ref{eq:def:G0+k}).  We also need the expression of $G_{1,+}^\bk$. We get
\begin{equation} \label{eq:def:G1+}
	G_{1, +}^\bk(\bx, \by) = \frac{-\ri}{2 | \WS | }  \left( 2 k_d x_d y_d + \bar{k} \cdot (\bar{x} - \bar{y})\left( | x_d - y_d | - | x_d + y_d | \right) \right) - \bar{k} \cdot g_{1, +} (\bx, \by),
\end{equation}
with
\[
	g_{1,+} := g_{1,\sharp}(\bar{x} - \bar{y}, x_d - y_d) - g_{1,\sharp}(\bar{x} - \bar{y}, x_d + y_d).
\]
From~\eqref{eq:relations_g1sharp}, we see that $g_{1,+}$ satisfies the symmetry relation
\begin{equation} \label{eq:symg1+}
	g_{1,+}(\bx, \by) = - g_{1,+} (\by, \bx).
\end{equation}

\medskip

Finally, from~\eqref{eq:magicRelations}, we deduce that
\begin{equation*} %\label{eq:magicRelations+}
	\Delta_\by G_{1,+}^\bk(\bx, \by) = \sum_{\bR \in \Lat} (-\ri \bar{k} \cdot \bR ) \delta_{\bR}(\bar{x} - \bar{y}) \left( \delta_0(x_d - y_d) - \delta_0(x_d + y_d) \right),
\end{equation*}
and that, for $n \in \N$, 
\begin{equation*}
	\Delta_\by G_{n+2,+}^\bk(\bx, \by)  + k^2 G_{n, +}^\bk(\bx, \by) =  \sum_{\bR \in \Lat} \dfrac{(-\ri \bar{k} \cdot \bR )^{n+2}}{(n+2)!}  \delta_{\bR}(\bar{x} - \bar{y}) \left( \delta_0(x_d - y_d) - \delta_0(x_d + y_d) \right).
\end{equation*}
In particular, if $\bx, \by \in \bar{D}$, then, since $\delta_\bR(\bar{x} - \bar{y}) = 0$ except for $\bR = \bnull$ (and $
\bar{x} = \bar{y}$), we have
\begin{equation} \label{eq:magicRelations+}
	\Delta_{\bx} G_{1,+}^\bk = \Delta_\by G_{1,+}^\bk = 0
	\quad \text{and} \quad
	\forall n \in \N, \quad \Delta_\by G_{n+2,+}^\bk(\bx, \by)  = - k^2 G_{n, +}^\bk(\bx, \by).
\end{equation}

%%%%%%%%%%%%%%%%%
\subsection{Periodic-Dirichlet layer potentials}

We now introduce the periodic-Dirichlet layer potential operators. 
We denote by $H^{-1/2} := H^{-1/2}(\p D)$ and by $H^{1/2} := H^{1/2}(\p D)$ the usual fractional Sobolev spaces on surfaces~\cite{Lions1972}. 
In the sequel, we use $\bra \cdot, \cdot \ket$ for the $H^{-1/2}, H^{1/2}$ duality pairing. We also introduce $H^{-1/2}_0 := \left\{ f \in H^{-1/2}, \ \bra f, 1 \ket = 0\right\}$.
%The adjoint of a linear operator $A$ shall be denoted by $A^*$. 
The periodic-Dirichlet single-layer potential $\cS_+^{\bk} : H^{-1/2} \to H^{1/2}$ and the periodic-Dirichlet Dirichlet-to-Neumann operator $\cK_+^{\bk, *} : H^{-1/2} \to H^{-1/2}$ are respectively defined, for smooth functions $\psi \in C^\infty(\p D)$ by 
%We denote by $L^2 := L^2(\p D)$ and by $H^1 := H^1(\p D)$ the usual Sobolev spaces on surfaces~\cite{Lions1972}. We recall that these spaces are Hilbert spaces, and we denote by $\bra \cdot, \cdot \ket$ the $L^2$ inner product. We also introduce $L^2_0 := \left\{ f \in L^2, \ \int_{\p D} f = 0\right\}$. The periodic-Dirichlet single-layer potential $\cS_+^k : L^2 \to H^1$ and the periodic-Dirichlet Dirichlet-to-Neumann operator $\cK_+^{k, *} : L^2 \to L^2$ are respectively defined, for smooth functions $\psi \in C^\infty(\p D)$ by 
\[
	\forall \bx \in \p D, \quad \cS_{+}^{\bk}[\psi](\bx) := \int_{\p D} G_+^{\bk}(\bx, \by) \psi(\by) \rd \sigma(\by),
	\quad
	\cK_{+}^{\bk,*}[\psi](\bx) := \int_{\p D} \dfrac{ \p G_+^\bk}{\p \nu_\bx}(\bx, \by) \psi(\by) \rd \sigma(\by).
\]
%and then extended on the whole space $H^{-1/2}$ by continuity and density. 
For simplicity, we write $\cK_+^* := \cK_+^{\bnull,*}$ and $\cS_+ := \cS_+^\bnull$. 
We also introduce the operator $\widetilde{\cS_{+}^\bk} : H^{-1/2} \to H^1_\loc(\R^d)$ defined by
\begin{equation} \label{eq:def:tildeSk}
	\forall \psi \in H^{-1/2}, \quad \forall \bx \in \R^d, \quad \widetilde{\cS_{+}^\bk}[\psi](\bx) := \int_{\p D} G_+^\bk(\bx, \by) \psi(\by) \rd \sigma(\by).
\end{equation}

% : E \to F$ is the operator $A^* : F' \to E'$, where $F'$ and $E'$ are the dual space of $F$ and $E$ respectively (for the $L^2(\p D)$ pivoting space) defined by
%\[
%	\forall f \in F', \forall g \in E, \quad \left\bra f, Ag \right\ket_{F', F} = \left\bra A^*f, g \right\ket_{E', E}.
%\]
%In the sequel, the adjoint of a linear operator $A : \cH^{+-} \to \cH^{--}_2$, where $\cH^{--}_{1,2}$ are Hillbert spaces, is the operator $A^* : \cH^{--}_2 \to \cH^{+-}$ defined by
%\[
%	\forall f \in \cH^{--}_2, \ \forall g \in \cH^{+-}, \quad \left\bra f, Ag \right\ket_{\cH^{--}_2} = \left\bra A^*f, g \right\ket_{\cH^{+-}}.
%\]
We have the following standard result.    
 
\begin{lemma} \label{lem:SkKk}~
\begin{enumerate}[(i)]
	\item For all $\bk \in \R^d$, the operator $\cS_{+}^{\bk} : H^{-1/2} \to H^{1/2}$ is an bounded operator with bounded inverse. 
		Moreover, it holds that $\cS_+^* = \cS_+$.% \commentaire{maybe there is a complex conjugate}.
	\item For all $\bk \in \R^d$, the operator $\cK_{+}^{\bk,*}$ is a compact operator on $H^{-1/2}$,
	and the operator $\cK_+^\bk$ is a compact operator on $H^{1/2}$.
	\item \textbf{(jump formulae)} It holds that
	\[
		\widetilde{\cS_{+}^\bk} \Big|_+ = \widetilde{\cS_{+}^\bk} \Big|_- = \cS_+^\bk \quad \text{and} \quad
		\left( \partial_\nu \widetilde{\cS_{+}^\bk} \right) \Big|_\pm = \pm \frac12 + \cK_+^{\bk,*}.
	\]
	\item It holds that $\sigma(\cK_{+}^*) = \sigma(\cK_+) \subset (-1/2, 1/2]$ and that
	\[
		\Ker \left( \cK_{+} - \frac12 \right) = \Span \{ \1_{\p D} \},
		\quad \text{and} \quad
		\Ker \left( \cK_{+}^* - \frac12 \right) = \Span \{ \psi_0^+ \},
	\]
	where $\1_{\p D} \in H^{1/2}$ is the constant function with value $1$ on $\p D$, and where $\psi_0^+ \in H^{-1/2}$ is such that $\bra \1_{\p D}, \psi_0^+ \ket = 1$. 
	\item The operator $(-\frac12 + \cK_+^*)$ acting on $H^{-1/2}_0$ is invertible with bounded inverse. 
	\end{enumerate}
\end{lemma}

We now introduce two constants. Recall that $\psi_0^+$ is defined by
$\Ker \left( \cK_{+}^* - \frac12 \right) = \Span \{ \psi_0^+ \}$ and $\bra \1_{\p D}, \psi_0^+ \ket = 1$. We define
\begin{eqnarray}
C_{D, \Lat}^+ &=& - \left\bra \cS_{+}^{-1} \1_{\p D}, \1_{\p D} \right\ket, \\
M_1 &=& \left\bra \psi_0^+, x_d |_{\p D} \right\ket .\label{eq:def:M1}
\end{eqnarray}

We have the following result.
\begin{lemma}
\begin{enumerate}[(i)]
\item
For all $\bx \in D$, $\widetilde{\cS_{+}}[\psi_0^+](\bx) = - 1/C_{D, \Lat}^+$. Especially, $\cS_{+}[\psi_0^+] = - (1/C_{D, \Lat}^+) \1_{\p D}$.

\item
Let $\psi_1^+ := \left( -\frac12 + \cK_+^* \right)^{-1} [\nu_d] \in H^{-1/2}_0$, where $\nu_d$ is the $d$-component of the outward normal $\nu$ to $\partial D$. Then it holds that 
	\[
		\forall \bx \in \p D,  \quad \cS_+[\psi_1^+](\bx) = x_d |_{\p D} - M_1 \quad \text{and} \quad \forall \bx \in D, \quad \widetilde{\cS_+}[\psi_1^+](\bx) = x_d - M_1,
	\]
	In particular, $\cS_+^{-1}[x_d |_{\p D}] = \psi_1^+ - C_{D, \Lat}^+ M_1 \psi_0^+$.
\end{enumerate}
\end{lemma}

\begin{proof}
	The proof of \textit{(i)} is straightforward. Let us prove \textit{(ii)}. For all $M \in \R$, the function $f(\bx) := x_d - M$ satisfies $\Delta f = 0$ and $\partial_\nu f |_{-} = \nu_d$. Together with the jump formulae, we deduce that
	\[
		\forall \bx \in D, \quad \widetilde{\cS_{+}}[\psi_1^+](\bx) = x_d - M,
	\]
	where $M$ is chosen so that $\psi_1^+ \in L^2_0$. To calculate $M$, we notice that
	\[
		0 = \bra \1_{\p D}, \psi_1^+ \ket = \bra \1_{\p D}, \cS_+^{-1}[x_d - M\1_{\p D}] \ket = \bra  \cS_+^{-1}[\1],x_d - M \1_{\p D}\ket = - C_{D, \Lat}^+ \bra \psi_0^+, x_d - M \1_{\p D} \ket.
	\]
	The result follows.
\end{proof}

\begin{definition} \label{def:periodicCapacity}
	We call the constant $C_{D, \Lat}^+$ the periodic capacity of $D$ with respect to the lattice $\Lat$. 
\end{definition}

%\begin{definition}
%	We denote by $M_2 := \int_D x_d \rd \bx$.
%\end{definition}

\begin{remark}
The periodic capacity $C_{D, \Lat}^+$ is positive. 
	%Its unit is\footnote{$m$ is for meter} $m^{d-1}$. Recall that usually the unit of a capacity is simply $m$. The unit of $M_1$ is $m$%, and the one of $M_2$ is $m^{d+1}$
Both $C_{D, \Lat}^+$ and $M_1$ depend on the lattice $\Lat$.
\end{remark}

%%%%%%%%%%%%%%%%%%%%%%%%%%%%%%%%%%%%%%%%%%%%%%%%%%%
\subsection{Equivalent formulation}

We now rescale the problem \eqref{eq:scattering}. Recall that $\delta = \delta_\varepsilon = \mu \varepsilon^2$. In the sequel, we denote by $\bX$ the macroscopic variable and by $\bx := \frac{\bX}{\varepsilon}$ the microscopic one. We denote by $u(\bx) := {U}(\varepsilon \bx)$. With this change of variable,~\eqref{eq:scattering} is equivalent to
\begin{equation} \label{eq:scatteringPb2}
	\left\{ \begin{aligned}
		& \left( \Delta + \varepsilon^2 k^2  \right) u^\varepsilon = 0 \quad \text{on} \quad \R^d-+ \setminus \overline{\Omega^1}, \\
		& \left( \Delta + \varepsilon^2 k_b^2  \right) u^\varepsilon = 0 \quad \text{on} \quad \Omega^1,\\
		& u^\varepsilon |+ = u^\varepsilon |_- \quad \text{on} \quad \p \Omega^1 ,\\
		& \partial_\nu u^\varepsilon |_- = \mu \varepsilon^2  \partial_\nu u^\varepsilon|_+  \quad \text{on} \quad \p \Omega^1 ,\\
		& u^\varepsilon - U^\rin_{\varepsilon \bk} \quad \text{satisfies the outgoing radiation condition},\\
		& u^\varepsilon = 0 \quad \text{on} \quad \p \R^2_+ .
		\end{aligned}
	\right.
\end{equation}

We use layer potentials to solve~\eqref{eq:scatteringPb2}. We consider the case when the incidence angle is such that $\bar{k}^2 \le k_b^2$. The case $\bar{k}^2 \ge k_b^2$ can be treated in a similar manner. 
We set $k_{b,d} = \sqrt{k_b^2 - \bar{k}^2}$, and denote by $\bk_b = (\bar{k}, k_{b,d}) \in \R^d$ the vector such that $| \bk_b | = k_b$. The solution to \eqref{eq:scatteringPb2} can be represented by
\begin{equation} \label{eq:ansatz}
	u^\varepsilon (\bx) = \left\{ \begin{aligned}
		& {U_{\varepsilon \bk}}^\rin(\bar{x}, x_d) - {U^\rin_{\varepsilon \bk}}(\bar{x}, -x_d) + \widetilde{\cS_{+}^{\varepsilon \bk}}[\psi](\bx) \quad \text{for} \quad \bx \in \R^d \setminus \overline{\Omega^1} , \\
		& \widetilde{\cS_{+}^{\varepsilon \bk_b}}[\psi_b](\bx) \quad \text{for} \quad \bx \in \Omega^1,
		\end{aligned}
		\right.
\end{equation}
where $\psi_b, \psi \in H^{-1/2}$ are surface potentials. In the sequel, we denote by $\cH^{--} := H^{-1/2} \times H^{-1/2}$ and by $\cH^{+-} := H^{1/2} \times H^{-1/2}$. After some straightforward calculations and using Lemma~\ref{lem:SkKk}, we see that~\eqref{eq:scatteringPb2} is equivalent to finding $\Psi(\varepsilon) = (\psi_b, \psi)^t \in \cH^{--}$ such that
\begin{equation} \label{eq:APsi=F}
	\cA(\varepsilon) \Psi(\varepsilon) = F(\varepsilon),
\end{equation}
where
\begin{equation} \label{eq:def:Aepsilon}
		\cA(\varepsilon) := \begin{pmatrix}
			\cS_{+}^{\varepsilon \bk_b} & - \cS_{+}^{\varepsilon \bk}\\
			-\frac12 + \cK_{+}^{\varepsilon \bk_b, *} & - \mu \varepsilon^2 \left( \frac12 + \cK_{+}^{\varepsilon \bk, *} \right)
		\end{pmatrix}
\end{equation}
 and
\begin{align} \label{eq:def:Fepsilon}
	F(\varepsilon) & := \begin{pmatrix}
		 U_{\varepsilon \bk}^\rin(\bar{x}, x_d) |_+ - {U^\rin_{\varepsilon \bk}}(\bar{x}, -x_d) |_+ \\
		\delta_\varepsilon \partial_\nu \left(  U_{\varepsilon \bk}^\rin(\bar{x}, x_d) - {U^\rin_{\varepsilon \bk}}(\bar{x}, -x_d) \right) |_+ \\
	\end{pmatrix} \nonumber \\
	& = 
	- 2 \ri u_0  \begin{pmatrix} 
		 \sin(\varepsilon k_d x_d)\\
		 \varepsilon^3 \mu \left( k_d \nu_d \cos(\varepsilon k_d x_d) - \ri \bar{k} \cdot \bar{\nu} \sin(\varepsilon k_d \nu_d) \right)
	\end{pmatrix}  \re^{- \ri \varepsilon \bar{k} \cdot \bar{x}}.
\end{align}

%%%%%%%%%%%%%%%%%%%%%%%%%%%%%%%%%%

By Lemma \ref{lem:SkKk}, $\cA(\varepsilon)$ is a bounded operator from $\cH^{--}$ to $\cH^{+-}$.  We study~\eqref{eq:APsi=F} using Taylor expansion. Following the decomposition~\eqref{eq:expansionGk+}, we write
\begin{equation} \label{eq:decompositionSkKk}
	\cS_{+}^{\varepsilon \bk} = \cS_{+} + \sum_{n=1}^\infty \varepsilon^n \cS_{n,+}^\bk
	\quad \text{and} \quad
	\cK_{+}^{\varepsilon \bk,*} = \cK_{+}^* + \sum_{n=1}^\infty \varepsilon^n \cK_{n,+}^{\bk,*},
\end{equation}
where the convergence holds in $\cB(H^{-1/2}, H^{1/2})$ and $\cB(H^{-1/2})$ respectively, and where, for $n \in \N^*$, $\psi \in H^{-1/2}$ and $\bx \in \p D$,
\[
	\cS_{n,+}^\bk[\psi](\bx) := \int_{\p D} G_{n, +}^\bk(\bx , \by) \psi(\by) \rd \sigma(\by) 
	\quad \text{and} \quad
	\cK_{n,+}^*[\psi](\bx) := \int_{\p D} \dfrac{\partial G_{n, +}^\bk}{\partial \nu_\bx}(\bx , \by) \psi(\by) \rd \sigma(\by) .
\]
Here, $\cB(H^{-1/2}, H^{1/2})$ denotes the set of linear bounded operators from $H^{-1/2}$ onto $H^{1/2}$ and 
$\cB(H^{-1/2})$ is the set of linear bounded operators on $H^{-1/2}$. 
Then we  write
\[
	\cA(\varepsilon) = \cA_0 + \cB(\varepsilon) \quad \text{with} \quad \cB(\varepsilon) := \sum_{n=1}^\infty \varepsilon^n \cA_n,
\]
where 
\[
	\cA_0 := \begin{pmatrix}
			\cS_{+} & - \cS_{+} \\
			-\frac12 + \cK_{+}^{*} & 0
		\end{pmatrix},
	\quad
	\cA_1 := \begin{pmatrix}
			 \cS_{1, +}^{\bk_b} & - \cS_{1, +}^\bk \\
			 \cK_{1, +}^{\bk_b, *} & 0
		\end{pmatrix},
	\quad 
	\cA_2 := \begin{pmatrix}
		 	\cS_{2, +}^{\bk_b} & -  \cS_{2, +}^\bk \\
			 \cK_{2, +}^{\bk_b, *} & -\mu \left( \frac12 + \cK_+^* \right)
		\end{pmatrix},
\]
and, for $n \ge 3$,
\[
	\cA_n := \begin{pmatrix}
		\cS_{n,+}^{\bk_b} & -  \cS_{n,+}^\bk \\
		\cK_{n,+}^{\bk_b,*} & - \mu \cK_{n-2, +}^{\bk,*}
	\end{pmatrix}.
\]
It is standard to check that the convergence holds in $\cB(\cH^{--}, \cH^{+-})$.
We would like to approximate $\cA(\varepsilon)^{-1}$ by $\cA_0^{-1}$. Unfortunately, this is not possible, since the operator $\cA_0$ is not invertible. It is indeed easy to check that $\Ker ( \cA_0 ) = \Span \left\{ \begin{pmatrix} \psi_0^+ \\ \psi_0^+ \end{pmatrix} \right\}$.
In order to handle this difficulty, we perturb the operator $\cA_0$ (see also the method used in~\cite{Ammari2016_Minnaert}). We introduce a rank-1 projection operator $\Pi \in \cB( H^{-1/2})$ defined by
\[
	\forall \psi \in H^{-1/2}, \quad \Pi[\psi] = \left\bra \1_{\p D}, \psi \right\ket \psi_0^+,
\]
and we set
\[
	\widetilde{\cA_0} := \cA_0 + \cP \quad \text{and} \quad \widetilde{\cA}(\varepsilon) := \cA(\varepsilon) + \cP,
	 \quad \text{with} \quad
	 \cP := \begin{pmatrix}  0 & 0 \\ \Pi & 0 \end{pmatrix}.
\]
%The following lemma is straightforward. % by noticing that the operator $\cA_0$ is a Fredholm operator with index $0$.
\begin{lemma}  \label{lem:widetildeA0}
The operator $\widetilde{\cA_0}$ is bounded and is invertible with inverse
%bounded inverse, and it holds that
\[
	\widetilde{\cA_0}^{-1} = \begin{pmatrix} 0 & (-\frac12 + \cK_+^* + \Pi)^{-1} \\  -\cS_+^{-1} & (-\frac12 + \cK_+^* + \Pi)^{-1} \end{pmatrix}.
\]
%\[
%	\widetilde{\cA_0}[\Psi_0]= \| \Psi_0 \|_\cH^{--}^2 \Phi_0  
%	\quad \text{and} \quad
%	\widetilde{\cA_0}^*[\Phi_0] = \| \Phi_0 \|_{\cH^{+-}}^2\Psi_0.
%\]
\end{lemma}

Recall that we want to calculate $\Psi$, the solution to~\eqref{eq:APsi=F}. We introduce
\[
	Q :=  \begin{pmatrix}  \Pi & 0 \\ 0 & 0 \end{pmatrix},
	\quad
	\Psi_- := (1 - Q) \Psi
	\quad \text{and} \quad
	\Psi_+ := Q \Psi
	\quad \text{so that} \quad
	\Psi = \Psi_- + \Psi_+.
\]
For the sake of clarity, we introduce the vectors $\Phi_1, \Phi_2 \in \cH^{--}$ and $\Phi_3, \Phi_4 \in \cH^{+-}$ defined by
\[
	\Phi_{1} := \begin{pmatrix} \psi_0^+ \\ 0 \end{pmatrix}, 
	\quad
	\Phi_{2} := \begin{pmatrix} 0 \\ \psi_0^+ \end{pmatrix},
	\quad
	\Phi_{3} := \begin{pmatrix} \1_{\p D} \\ 0 \end{pmatrix}, 
	\quad \text{and} \quad
	\Phi_{4} := \begin{pmatrix} 0 \\ \1_{\p D} \end{pmatrix}.
\]
Note that $\Psi_+ = \alpha \Phi_1$ for some $\alpha \in \C$. The equation~\eqref{eq:APsi=F} is therefore equivalent to
\begin{equation} \label{eq:equiv1}
	\left( \widetilde{\cA_0} + \cB(\varepsilon) - \cP \right) \left( \alpha \Phi_1 + \Psi_- \right) = F(\varepsilon).
\end{equation}
From Lemma~\ref{lem:widetildeA0}, we deduce that for $\varepsilon$ small enough, the operator $\widetilde{\cA_0} + \cB(\varepsilon)$ is invertible, and that its inverse is given by the Neumann series
\begin{equation} \label{eq:Neumann}
	\widetilde{\cA}(\varepsilon)^{-1} = \left( \widetilde{\cA_0} + \cB(\varepsilon) \right)^{-1} = \widetilde{\cA_0}^{-1} + \sum_{n=1}^\infty (-1)^n \left( \widetilde{\cA_0}^{-1} \cB(\varepsilon) \right)^n \widetilde{\cA_0}^{-1},
\end{equation}
where the convergence holds in $\cB(\cH^{+-}, \cH^{--})$. Applying $\left( \widetilde{\cA_0} + \cB(\varepsilon) \right)^{-1}$ to both sides of~\eqref{eq:equiv1}, and using the fact that $\cP \Phi_1 = \Phi_2$ and that $\cP \Psi_- = 0$, we obtain
\[
	\alpha \Phi_1 + \Psi_- - \alpha \left( \widetilde{\cA_0} + \cB(\varepsilon) \right)^{-1} \Phi_2 = \left( \widetilde{\cA_0} + \cB(\varepsilon) \right)^{-1} F(\varepsilon).
\]
Finally, we notice that $\bra \Phi_3, \Psi_- \ket = 0$, so that, by taking the duality product with $\Phi_3$, we obtain that~\eqref{eq:APsi=F} is equivalent to
\begin{equation} \label{eq:system}
	\left\{ \begin{aligned}
		& \alpha = \dfrac{\left\bra \Phi_3, \left( \widetilde{\cA_0} + \cB(\varepsilon) \right)^{-1} F(\varepsilon) \right\ket}{1 - \left\bra \Phi_3, \left( \widetilde{\cA_0} + \cB(\varepsilon) \right)^{-1} \Phi_2 \right\ket}, \\
		& \Psi_- = \left( \widetilde{\cA_0} + \cB(\varepsilon) \right)^{-1} F(\varepsilon) - \alpha \Phi_1 + \alpha \left( \widetilde{\cA_0} + \cB(\varepsilon) \right)^{-1} \Phi_2.
	\end{aligned} \right.
\end{equation}

%%%%%%%%%%%%%%%%%%%%%%%%%%%%%%%%%%

\subsection{Asymptotic expansions}

We now solve~\eqref{eq:system} using asymptotic expansions in $\varepsilon$. We first need some estimates.
\begin{lemma} \label{lem:estimates}
	The following identities hold.
	\begin{align*}
	& (i) \ \cK_{1,+}^\bk [\1_{\p D}] = 0, \quad \\
	& (ii) \ \cK_{2,+}^\bk [\1_{\p D}](\bx) = - k^2 \int_{D} G_{+}^{\bnull}(\bx, \by) \rd \by, \quad \\
	& (iii) \ \cK_{3,+}^\bk [\1_{\p D}](\bx) = - k^2 \int_{D} G_{1,+}^\bk(\bx, \by) \rd \by.
	\end{align*}
\end{lemma}
\begin{proof}
First, we recall that $G_{1,+}^\bk (\bx, \by)$ satisfies~\eqref{eq:magicRelations+}. In particular,
\[
	(i) \quad
	\cK_{1, +}^\bk [\1_{\p D}](\bx)
		=   \int_{\p D} \partial_{\nu_\by} G_{1, +}^\bk(\bx, \by) \rd \sigma(\by)
		= \int_D \Delta_{\bx} G_{1, +}^\bk(\bx, \by) \rd \bx = 0.
\]
Also, using~\eqref{eq:magicRelations+}, we obtain
\begin{align*}
	(ii) \quad \cK_{2,+}^\bk [\1_{\p D}](\bx) & = \int_{\p D} \p_{\nu_{\by}} G_{2,+}^\bk(\bx, \by)  \rd \sigma(\by)
		= \int_D \Delta_{\by} G_{2,+}^\bk (\bx, \by) \rd \by
		= - k^2 \int_D G_{+}^\bnull (\bx, \by) \rd \by,
\end{align*}
and similarly,
\begin{align*}
	(iii) \quad \cK_{3,+}^\bk [\1_{\p D}](\bx) & = \int_{\p D} \p_{\nu_{\by}} G_{3,+}^\bk(\bx, \by)  \rd \sigma(\by)
		=- k^2 \int_D G_{1,+}^\bk(\bx, \by) \rd \by.
\end{align*}

\end{proof}
We start with the calculation of $\alpha$.

\medskip

%%%%%%%%%%%%%%%%%%%%%%%%%%%%%%%%%%
\noindent \textbf{$\bullet$ Evaluation of $\left\bra \Phi_3, \left( \widetilde{\cA_0} + \cB(\varepsilon) \right)^{-1} \Phi_2 \right\ket$.}
From~\eqref{eq:Neumann}, we have
\begin{equation} \label{eq:asympAm1}
\begin{aligned}
	\left( \widetilde{\cA_0} + \cB(\varepsilon) \right)^{-1} 
		& = \widetilde{\cA_0}^{-1} - \varepsilon \left( \widetilde{\cA_0}^{-1} \cA_1 \widetilde{\cA_0}^{-1} \right)
		+ \varepsilon^2 \left( \left[\widetilde{\cA_0}^{-1} \cA_1 \right]^2 \widetilde{\cA_0}^{-1} - \widetilde{\cA_0}^{-1} \cA_2 \widetilde{\cA_0}^{-1}  \right)\\
		& - \varepsilon^3 \left( \left[\widetilde{\cA_0}^{-1} \cA_1 \right]^3 \widetilde{\cA_0}^{-1} + \widetilde{\cA_0}^{-1} \cA_3 \widetilde{\cA_0}^{-1} 
		\right) \\
		& + \varepsilon^3 \left( \widetilde{\cA_0}^{-1}\cA_1 \widetilde{\cA_0}^{-1} \cA_2 \widetilde{\cA_0}^{-1} + \widetilde{\cA_0}^{-1} \cA_2 \widetilde{\cA_0}^{-1} \cA_1 \widetilde{\cA_0}^{-1} \right) + O(\varepsilon^4).
\end{aligned}
\end{equation}
It holds that $\widetilde{\cA_0}^{-1} \Phi_2 = \Phi_1 + \Phi_2$ and $\left(\widetilde{\cA_0}^{-1}\right)^* \Phi_3 = \Phi_4$. Moreover, using Lemma~\ref{lem:estimates} we have 
\[
	\left( \widetilde{\cA_0}^{-1} \cA_1 \right)^* \Phi_3 = \cA_1^* \Phi_4 
	= \begin{pmatrix} \cK_{1,+}^{\bk_b} [\1_{\p D}] \\ 0 \end{pmatrix} = 0.
\]
As a consequence,~\eqref{eq:asympAm1} simplifies into
\begin{equation} \label{eq:asympAm1_v2}
\begin{aligned}
	\left\bra \Phi_3, \left( \widetilde{\cA_0} + \cB(\varepsilon) \right)^{-1} \Phi_2 \right\ket 
	 = & 1 - \varepsilon^2 \left\bra \Phi_4, \cA_2 (\Phi_1 + \Phi_2) \right\ket - \varepsilon^3 \left\bra \Phi_4, \cA_3 (\Phi_1 + \Phi_2) \right\ket \\
	& + \varepsilon^3 \left\bra \Phi_4, \cA_2 \widetilde{\cA_0}^{-1} \cA_1 (\Phi_1 + \Phi_2) \right\ket + O(\varepsilon^4).
\end{aligned}
\end{equation}
Let us evaluate the terms appearing in~\eqref{eq:asympAm1_v2}. First, it holds that
\begin{align*}
	\left\bra \Phi_4, \cA_2 (\Phi_1 + \Phi_2) \right\ket & 
		= \left\bra \1_{\p D}, \cK_{2,+}^{\bk_b,*} [\psi_0^+] \right\ket - \mu \left\bra \1_{\p D}, \left( \frac12 + \cK_+^* \right) \psi_0^+ \right\ket.
\end{align*}
Recall that $\cK_+^* [ \psi_0^+] = \frac12 \psi_0^+$. By Lemma~\ref{lem:estimates}, we have
\[
	\left\bra \cK_{2,+}^{\bk_b} [\1_{\p D}], \psi_0^+ \right\ket = - k_b^2 \int_{\p D} \int_D G_+^\bnull(\bx, \by) \psi_0^+(\bx) \rd \by \rd \sigma(\bx) = - k_b^2 \int_D \widetilde{S_+}[\psi_0^+](\by) \rd \by = \frac{k_b^2 | D | }{C_{D, \Lat}^+}.
\]
Hence, 
\[
	\left\bra \Phi_4, \cA_2 (\Phi_1 + \Phi_2) \right\ket = \dfrac{k_b^2 | D | }{C_{D, \Lat}^+} - \mu = \mu_M - \mu , 
\]
where $\mu_M$ was defined in~\eqref{eq:def:muM}. Similarly, using Lemma~\ref{lem:estimates}, one gets
\[
	\left\bra \Phi_4, \cA_3  (\Phi_1 + \Phi_2) \right\ket =   \left\bra \1_{\p D}, \cK_{3,+}^{\bk_b,*} [\psi_0^+] \right\ket 
	= - k_b^2 C_1,
\]
where for simplicity we set
\begin{equation} \label{eq:def:C1}
	 C_1 := \int_D \int_{\p D} G_{1, +}^{\bk_b}(\bx, \by) \psi_0^+(\bx) \rd \sigma(\bx) \rd \by.
\end{equation}
Finally, it remains to evaluate $\left\bra \Phi_4, \cA_2 \widetilde{\cA_0}^{-1} \cA_1  (\Phi_1 + \Phi_2) \right\ket$. First, we have that
\[
	\cA_1 [\Phi_1 + \Phi_2] =  \begin{pmatrix} \left( \cS_{1,+}^{\bk_b} - \cS_{1,+}^\bk \right)[\psi_0^+] \\
		\cK_{1,+}^{\bk_b,*}[\psi_0^+] \end{pmatrix}.
\]
By noticing that $\bar{k_b} = \bar{k}$, together with the expression~\eqref{eq:def:G1+}, we see that the contribution of $\bar{k}$ in $\cS_{1,+}^{\bk_b} - \cS_{1,+}^\bk$ cancels. Hence,
\[
	 \left( \cS_{1,+}^{\bk_b} - \cS_{1,+}^\bk \right)[\psi_0^+] = \left( k_{b,d} - k_d \right) \int_{\p D} \dfrac{- \ri x_d y_d}{| \WS |} \psi_0^+(\by) \rd \sigma(\by) =  \dfrac{- \ri M_1}{| \WS |} \left( k_{b,d} - k_d \right) x_d.
\]
In particular,
\[
	\widetilde{\cA_0}^{-1} \cA_1 [\Phi_1 + \Phi_2] = \begin{pmatrix}
		(-\frac12 + \cK_+^* + \Pi) \cK_{1,+}^{\bk_b,*}[\psi_0^+] \\
		(-\frac12 + \cK_+^* + \Pi)^{-1} \cK_{1,+}^{\bk_b,*}[\psi_0^+] - \dfrac{- \ri M_1}{| \WS |} \left( k_{b,d} - k_d \right) \left( \psi_1^+ - M_1 C_{D, \Lat}^+ \psi_0^+ \right)
		\end{pmatrix} . 
\]
%We notice that $(-\frac12 + \cK_+^* + \Pi)^{-1} [\nu_d] = \psi_1^+$ and recall that $\cS_+^{-1}[x_d] = \psi_1^+ - M_1 C_{D,\Lat}^+ \psi_0^+$, so that \commentaire{TO CHECK}
%\[
%	\widetilde{\cA_0}^{-1} \cA_1 [\Phi_1 + \Phi_2] = 
%	\frac{-\ri M_1}{ | \WS |}  \widetilde{\cA_0}^{-1}  \begin{pmatrix} (k_b - k) x_d \\ k_b \nu_d \end{pmatrix}
%	=  \frac{-\ri M_1}{ | \WS | }  \begin{pmatrix} k_b \psi_1^+ \\ k \psi_1^+ + (k_b - k) M_1 C_{D, \Lat}^+ \psi_0^+ \end{pmatrix}.
%\]
On the other hand, we have (using the fact that $\cK_+ [\1_{\p D}] = \frac12 \1_{\p D}$)
\begin{equation} \label{eq:A2starPhi4}
	\cA_2^* \Phi_4 = \begin{pmatrix} \cK_{2,+}^{\bk_b}[\1_{\p D}]  \\ -\mu  \left( \frac12 + \cK_+ \right) \1_{\p D} \end{pmatrix} = 
	\begin{pmatrix} - k_b^2 \int_D G_+^\bnull(\bx, \by) \rd \by \\ -\mu \1_{\p D} \end{pmatrix}.
\end{equation}
%so that,
%\[
%	\widetilde{\cA_0^{-*}} \cA_2^* \Phi_4 = 
%	\begin{pmatrix} 
%	- \mu \cS_{+}^{-*} [\1_{\p D}] \\ 
%	- k_b^2 \left( - \frac12 + \cK_+ + \Pi^* \right) \int_D G_+^\bnull (\bx, \by) \rd \by - \mu \1_{\p D}\end{pmatrix}.
%\]
Altogether, we obtain that 
%\[
%	\left\bra \Phi_4, \cA_2 \widetilde{\cA_0}^{-1} \cA_1  (\Phi_1 + \Phi_2) \right\ket = 
%	 \frac{-\ri M_1}{ | \WS |} \left( - k_b^3 \left\bra \int_D G_+^0(\bx, \by) \rd \by , \psi_1^+ \right\ket - \mu (k_b - k) M_1 C_{D, \Lat}^+ \right)
%\]
\[
	\left\bra \Phi_4, \cA_2 \widetilde{\cA_0}^{-1} \cA_1  (\Phi_1 + \Phi_2) \right\ket = -k_b^2 \left\bra \int_D G_+^\bnull(\bx, \by) \rd \by,  \left( -\frac12 + \cK_+^* + \Pi \right)^{-1} \cK_{1,+}^{\bk_b,*}[\psi_0^+] \right\ket + \frac{\ri \mu M_1^2 C_{D, \Lat}^+}{| \WS |} (k_{b,d} - k_d).
\]
Let us compute the inner product. We introduce the map $H : \bar{D} \to \C$ defined by
\[
	\forall \bx \in \bar{D}, \quad H(\bx) := \int_{\p D} G_{1,+}^{\bk_b}(\bx, \by) \psi_0^+(\by) \rd \sigma(\by),
\]
and we set $T = \cS_+^{-1}[H |_-]$ and $I = \int_{\p D} T \rd \sigma$. Thanks to~\eqref{eq:magicRelations+}, we see that $\Delta H = 0$ on $D$. Together with the jump relation formulae (see Lemma~\ref{lem:SkKk}), we deduce that
\[
	\left( -\frac12 + \cK_+^* \right)[T] = \dfrac{\partial H}{\partial_{\nu_\bx}} = \cK_{1,+}^{\bk_b,*}[\psi_0^+].
\]
Therefore \[
	\left( -\frac12 + \cK_+^* + \Pi \right)^{-1}  \cK_{1,+}^{\bk_b,*}[\psi_0^+] = T -I \psi_0^+.
\]
It follows that
\begin{align*}
	& \left\bra \int_D G_+^\bnull(\bx, \by) \rd \by,  \left( -\frac12 + \cK_+^* + \Pi \right)^{-1} \cK_{1,+}^{\bk_b,*}[\psi_0^+] \right\ket 
	= \int_D  \left( \int_{\p D}G_+^\bnull(\bx, \by) \left( T -I \psi_0^+ \right)(\by) \rd \sigma(\by) \right) \rd \bx \\
	& \qquad = \int_D \widetilde{\cS_+} \left[ T - I \psi_0^+ \right](\bx) \rd \bx = \int_D H(\bx) \rd \bx +I \dfrac{| D |}{C_{D, \Lat}^+} = C_1 +I \dfrac{| D |}{C_{D, \Lat}^+},
\end{align*}
where $C_1$ was defined in~\eqref{eq:def:C1}. It remains to compute the constant $I$. We have
\begin{align*}
	I & = \bra \1_{\p D}, T \ket = \bra \1_{\p D}, \cS_+^{-1} [H ] \ket =  \bra  \cS_+^{-1}[\1_{\p D}], H \ket = - C_{D, \Lat}^+ \bra \psi_0^+, H \ket \\
	&  = - C_{D, \Lat}^+ \int_{\p D} \int_{\p D} G_{1,+}^{\bk_b}(\bx, \by) \psi_0^+(\bx) \psi_0^+(\by) \rd \sigma(\bx) \rd \sigma(\by) \\
	& = \frac{- C_{D, \Lat}^+}{2}  \int_{\p D} \int_{\p D} \psi_0^+(\bx) \psi_0^+(\by) \left( G_{1,+}^{\bk_b}(\bx, \by) + G_{1,+}^{\bk_b}(\by, \bx) \right) \rd \sigma(\bx) \rd \sigma(\by),
\end{align*}
where we performed the change of variable $(\bx, \by) \to (\by, \bx)$ to obtain the last equality. From the expression of $G_{1,+}^{\bk_b}$ in~\eqref{eq:def:G1+} and the symmetry relation~\eqref{eq:symg1+}, we get
\[
	I = \dfrac{\ri C_{D, \Lat}^+ }{| \WS |} k_{b,d}  \int_{\p D}  x_d \psi_0(\bx) \rd \sigma(\bx)  \int_{\p D}  y_d \psi_0(\by) \rd \sigma(\by) = \frac{\ri M_1^2 C_{D, \Lat}^+}{| \WS |} k_{b,d}.
\]

%with (recall that for $\bx \in D$, we have $\widetilde{\cS_+}[\psi_1^+](\bx) = x_d - M_1$)
%\[
%	 \left\bra \int_D G_+^0(\bx, \by) \rd \by , \psi_1^+ \right\ket = \int_{D} \int_{\p D} G_+^0(\bx, \by)  \psi_1^+(\by) \rd \sigma (\by) \rd \bx 
%	 = \int_D \widetilde{\cS_+}[\psi_1^+] (\bx) \rd \bx
%	 = M_2 - | D | M_1.
%\]

Altogether, the above calculations yield
\begin{align*}
	1 - \left\bra \Phi_3, \left( \widetilde{\cA_0} + \cB(\varepsilon) \right)^{-1} \Phi_2 \right\ket_{\cH^{--}} 
	& = \varepsilon^2 \left(  \mu_M- \mu \right)  + \varepsilon^3 \left(-k_b^2 C_1 \right) \\
	& \quad - \varepsilon^3 \left( -k_b^2 \left[ C_1 +  \dfrac{\ri M_1^2 | D |}{| \WS |} k_{b,d}\right] + \dfrac{\ri \mu M_1^2 C_{D, \Lat}^+}{| \WS |}(k_{b,d} - k_d) \right) + O(\varepsilon^4) \\
	& =  \varepsilon^2 \left(  \mu_M- \mu \right)  + \varepsilon^3  \ri \frac{M_1^2}{| \WS |} \left(  k_b^2 k_{b,d}  | D |  - \mu (k_{b,d} - k_d) C_{D, \Lat}^+ \right) + O(\varepsilon^4) \\
	& = \varepsilon^2 (\mu_M - \mu) \left( 1 + \varepsilon \ri \frac{M_1^2 k_{b,d} C_{D, \Lat}^+}{| \WS |} \right) + \varepsilon^3 \ri \mu \frac{M_1^2 k_d C_{D, \Lat}^+ }{| \WS |} + O(\varepsilon^4).
\end{align*}

%We can keep this last expression, or further manipulate it. In the sequel, we will rather use the (weaker but simpler) estimate
%\begin{align*}
%	1 - \left\bra \Phi_3, \left( \widetilde{\cA_0} + \cB(\varepsilon) \right)^{-1} \Phi_2 \right\ket_{\cH^{--}} 
%	& = \varepsilon^2 \left(  \left(  \mu_M- \mu \right)  + \varepsilon  \ri \frac{M_1^2}{| \WS |} \left(  k_b^2 k_{b,d}  | D |  - \mu (k_{b,d} - k_d) C_{D, \Lat}^+ \right) \right) \left( 1 + O(\varepsilon) \right) \\
%	& = \varepsilon^2 \left(  \left(  \mu_M- \mu \right)  + \varepsilon  \ri \frac{M_1^2}{| \WS |} \left(  k_b^2 k_{b,d}  | D |  - \mu_M (k_{b,d} - k_d) C_{D, \Lat}^+ \right) \right) \left( 1 + O(\varepsilon) \right) \\
%	& =  \varepsilon^2 \left(  \left(  \mu_M- \mu \right)  + \varepsilon  \ri  \frac{M_1^2 k_d}{ | \WS |} \mu_M    \right) \left( 1 + O(\varepsilon) \right)\\
%	& =  \varepsilon^2 \left(  \left(  \mu_M- \mu \right)  + \varepsilon  \ri  \frac{M_1^2 k_d}{ | \WS |}\mu    \right) \left( 1 + O(\varepsilon) \right).
%\end{align*}

%%%%%%%%%%%%%%%%%%%%%%%%%%%%%%%%%%
\noindent \textbf{$\bullet$ Evaluation of $\left\bra \Phi_3, \left( \widetilde{\cA_0} + \cB(\varepsilon) \right)^{-1} F(\varepsilon) \right\ket$.} 
From~\eqref{eq:def:Fepsilon}, we obtain that
\[
	F = \varepsilon F_1 + \varepsilon^2 F_2 + \varepsilon^3 F_3+ O(\varepsilon^4),
\]
with
\[
	F_1 := - 2 \ri  u_0 \begin{pmatrix} k_d x_d \\ 0\end{pmatrix},
	\quad 
	F_2 := - 2 \ri u_0 \begin{pmatrix} - \ri (\bar{k} \cdot \bar{x}) k_d x_d \\ 0 \end{pmatrix},
	\quad \text{and} \quad
	F_3 := - 2 \ri u_0  \begin{pmatrix} 
		 \frac16  (k_d x_d)^3 - \frac12 (\bar{k} \cdot \bar{x})^2 k_d x_d \\
		  \mu k_d  \nu_d
	\end{pmatrix}.
\]
Using the decomposition~\eqref{eq:asympAm1} and similar estimates as before, we get
\begin{align*}
	\left\bra \Phi_3, \left( \widetilde{\cA_0} + \cB(\varepsilon) \right)^{-1} F(\varepsilon) \right\ket
		= & \
	\varepsilon \left\bra \Phi_3, \widetilde{\cA_0}^{-1} F_1 \right\ket + \varepsilon^2 \left\bra \Phi_3, \widetilde{\cA_0}^{-1} F_2 \right\ket \\
	 & \quad - \varepsilon^3 \left\bra \Phi_3, \widetilde{\cA_0}^{-1} \cA_2 \widetilde{\cA_0}^{-1} F_1 \right\ket + \varepsilon^3 \left\bra \Phi_3, \widetilde{\cA_0}^{-1} F_3 \right\ket + O(\varepsilon^4).		
\end{align*}
We have that $\left\bra \Phi_3, \widetilde{\cA_0}^{-1} F_1 \right\ket = \left\bra \Phi_3, \widetilde{\cA_0}^{-1} F_2 \right\ket = \left\bra \Phi_3, \widetilde{\cA_0}^{-1} F_3 \right\ket  = 0$. Also, using~\eqref{eq:A2starPhi4}, we have
\[
	\left\bra \Phi_3, \widetilde{\cA_0}^{-1} \cA_2 \widetilde{\cA_0}^{-1} F_1 \right\ket = \left\bra \Phi_4, \cA_2 \widetilde{\cA_0}^{-1} F_1 \right\ket
	=  - 2 \ri u_0 k_d \mu \left\bra   \1_{\p D} ,  \psi_1^+ - M_1 C_{D, \Lat}^+ \psi_0^+ \right\ket = 2 \ri u_0 k_d \mu M_1 C_{D, \Lat}^+ .
\]
We can conclude that
\[
	\left\bra \Phi_3, \left( \widetilde{\cA_0} + \cB(\varepsilon) \right)^{-1} F(\varepsilon) \right\ket = - \varepsilon^3 2 \ri u_0 k_d \mu M_1 C_{D, \Lat}^+ + O(\varepsilon^4). %= - \varepsilon^3 2 \ri u_0 k_d \mu M_1 C_{D, \Lat}^+ \left( 1 + O(\varepsilon) \right)
\]

%%%%%%%%%%%%%%%%%%%%%%%%%%%%%%%%%%
\noindent \textbf{$\bullet$ Evaluation of $\alpha$.}
From the above calculations, we deduce that 
\[
	\alpha = - \dfrac{ \varepsilon 2 \ri u_0 \mu k_d M_1 C_{D, \Lat}^+ + O(\varepsilon^2)}{(\mu_M - \mu) \left( 1 + \varepsilon \ri \frac{M_1^2 k_{b,d} C_{D, \Lat}^+}{| \WS |} \right) + \varepsilon \ri \mu \frac{M_1^2 k_d C_{D, \Lat}^+ }{| \WS |} + O(\varepsilon^2)}.
\]
In order to simplify this expression, we recall the scattering function $g_s(\mu, \varepsilon)$ defined in~\eqref{eq:def:gs}. We can check that 
\[
	\alpha = \left(2 \ri u_0 k_d  C_{D, \Lat}^+ + O(\varepsilon) \right)\cdot g_s(\mu, \varepsilon),
\]
in the sense
\[
	\exists C \in \R^+, \ \forall \varepsilon \ge 0, \ \forall \mu > 0, \quad  \left| \alpha -  (2 \ri u_0 k_d)  C_{D, \Lat}^+ g_s(\mu, \varepsilon)  \right| \le C \varepsilon | g_s(\mu, \varepsilon) |.
\]

%%%%%%%%%%%%%%%%%%%%%%%%%%%%%%%%%%
\noindent \textbf{$\bullet$ Evaluation of $\Psi_-$.} We finally calculate $\Psi_-$ defined in the second equation of~\eqref{eq:system}. Using similar calculations as before, we get
\begin{align}
	\Psi_- & = \varepsilon \widetilde{\cA_0}^{-1} F_1 - \alpha \Phi_1 + \alpha \widetilde{\cA_0}^{-1} \Phi_2 + O(\varepsilon^2 + | \varepsilon \alpha |) \nonumber \\
		& =  (2 \ri u_0 k_d) \left( \varepsilon \begin{pmatrix} 0 \\ \psi_1^+ - M_1 C_{D, \Lat}^+ \psi_0^+ \end{pmatrix} + g_s(\mu, \varepsilon)   C_{D, \Lat}^+ \begin{pmatrix} 0 \\ \psi_0^+ \end{pmatrix} \right) + O( \varepsilon | g_s(\mu, \varepsilon) |). \label{eq:psi-}
\end{align}
%Equivalently, 
%\begin{equation} \label{eq:psi-}
%	\Psi_- =  \varepsilon (2 \ri u_0 k_d) \left[  \begin{pmatrix} 0 \\ \psi_1^+  - M_1 C_{D, \Lat}^+ \psi_0^+ \end{pmatrix} +
%		  \dfrac{ M_1 C_{D, \Lat}^+}{ \left(   1 - \frac{\mu_M}{\mu} \right)  -  \varepsilon  \ri  \frac{M_1^2 k_d}{ | \WS |} }
%	\begin{pmatrix} 0 \\ \psi_0^+ \end{pmatrix} \right]  \left( 1+ O(\varepsilon) \right).
%\end{equation}

%%%%%%%%%%%%%%%%%%%%%%%%%%%%%%%%%%

\subsection{Microscopic scattered field}

Recall \eqref{eq:ansatz}, we have 
\[
	u_\varepsilon(\bx) = - 2 \ri u_0 \sin(\varepsilon k_d x_d) \re^{- \ri \varepsilon \bar{k} \cdot \bar{x}} + u^s_\varepsilon(\bx), 
\]
where $u_\varepsilon^s = \widetilde{\cS_+^{\varepsilon k}} [\psi_\varepsilon]$
is the scattered field. Note that $\Psi_+$ do not contribute to the field outside the bubbles. Using~\eqref{eq:psi-}, we obtain that
\begin{equation} \label{eq:def:us}
	u^s_\varepsilon(\bx) = (2 \ri u_0 k_d) \int_{D} G_+^{\varepsilon \bk}(\bx , \by) \left( \varepsilon \left[ \psi_1^+ - M_1 C_{D, \Lat}^+ \psi_0^+ \right] +
	g_s(\mu, \varepsilon)  C_{D, \Lat}^+ \psi_0^+ \right)(\by) \rd \sigma(\by) + O( \varepsilon | g_s(\mu, \varepsilon) |).
\end{equation}
We now simply the above integral by exploiting decomposition of $G_+^{\varepsilon \bk}$. 
%where the correct functional space for the previous estimate to hold will be detailed later.  
Recall that $L > 0$ is chosen such that, for all $\by \in D$, it holds that $0 \le y_d \le L$. Together with~\eqref{eq:def:Gk+}, we obtain that, for $x_d \ge L$, we have $G_+^{\bk}  = G_{+,p}^{\bk} + G_{+,e}^{\bk}$, where
\begin{eqnarray}
	 G_{+,p}^{\bk} (\bx, \by) &=& \left( \frac{\re^{\ri k_d (x_d - y_d)}}{2 \ri k_d | \WS |} - \frac{\re^{\ri k_d (x_d + y_d)}}{2 \ri k_d | \WS |} \right) \re^{- \ri \bar{k} \cdot \bar{x}}
	=- \frac{\sin(k_d y_d)}{k_d | \WS |} \re^{\ri ( k_d  x_d- \bar{k} \cdot \bar{x})}, \\
	G_{+,e}^{\bk}(\bx, \by) &=& - \sum_{\bl \in \RLat \setminus \{ \bnull \}} \left( \dfrac{\re^{\ri (\bl - \bar{k} ) \cdot (\bar{x} - \bar{y})}}{| \WS | \sqrt{ | \bl - \bar{k} |^2  - k^2}} \sinh \left( \sqrt{ | \bl - \bar{k} |^2  -  k^2} y_d \right) \right) \re^{-  \sqrt{ | \bl - \bar{k} |^2  -  k^2} x_d}. \label{eq:def:G+e}
\end{eqnarray}
It is clear that $G_{+,p}^{\bk}$ consists of the propagative mode,  while $G_{+,e}^{\bk}$ consists of the evanescent modes which are exponentially decaying away from the plane.

Accordingly, we write $\widetilde{\cS_+^{ \bk}} = \widetilde{\cS_{+,p}^{ \bk}} + \widetilde{\cS_{+,e}^{ \bk}}$, with
\begin{eqnarray*}
	 \widetilde{\cS_{+,p}^{ \bk}}[\psi](\bx) &=& \int_{\p D} G_{+,p}^{\bk} (\bx, \by) \psi(\by) \rd \sigma(\by), \\
	  \widetilde{\cS_{+,e}^{ \bk}}[\psi](\bx) &=& \int_{\p D} G_{+,e}^{\bk} (\bx, \by) \psi(\by) \rd \sigma(\by).
\end{eqnarray*}
In particular, in the case when $\bk=0$, we can deduce using ~\eqref{eq:G0expansion} that
\begin{eqnarray} 
\widetilde{S_{+,p}^{\bk = \bnull}}[\psi](\bx)  &= & -\dfrac{1}{| \WS |}\int_{\p D} y_d \psi(\by) \rd \sigma(\by), \\
	 \widetilde{S_{+,e}^{\bk = \bnull}}[\psi](\bx) &= & \widetilde{S_{+}}[\psi](\bx) + \dfrac{1}{| \WS |}\int_{\p D} y_d \psi(\by) \rd \sigma(\by), \label{eq:Se}
\end{eqnarray}
for $\psi \in H^{-1/2}$ and $x_d \ge L$.

Following  (\ref{eq:def:G0+k}) and (\ref{eq:def:Gk+}),   we set
\[
	\widetilde{S_{+,p,0}^{\bk}}[\psi](\bx) := - \re^{\ri (k_d x_d - \bar{k} \cdot \bar{x})} \frac{1}{| \WS | } \int_{\p D} y_d \psi(\by) \rd \sigma(\by).
\]
and
\[
	\widetilde{S_{+,e,0}^{\bk}}[\psi](\bx) := \re^{-\ri \bar{k} \cdot \bar{x}} \left( \widetilde{S_{+}}[\psi](\bx) + \dfrac{1}{| \WS |}\int_{\p D} y_d \psi(\by) \rd \sigma(\by) \right).
\]
The operators~$\widetilde{S_{+,p,0}^{\varepsilon \bk}}$ and $\widetilde{S_{+,e,0}^{\varepsilon \bk}}$ are good approximations for $\widetilde{S_{+,p}^{\varepsilon \bk}}$ and $\widetilde{S_{+,e}^{\varepsilon \bk}}$ respectively. Actually, 
from the expression of $G_{+,p}^{\bk}$ and $G_{+,e}^{\bk}$, we have the following results.
\begin{lemma} \label{lem:GFepsilon}
There exists $C \in \R^+$ such that, for all $\varepsilon \ge 0$, all $x_d \ge L$ and all $\psi \in H^{-1/2}$, it holds that
	\[
		\left| \left( \widetilde{S_{+,p}^{\varepsilon \bk}} - \widetilde{S_{+,p,0}^{\varepsilon \bk}} \right) [\psi] \right| (\bx) + \frac{1}{\varepsilon} \left| \nabla \left(\widetilde{S_{+,p}^{\varepsilon \bk}} - \widetilde{S_{+,p,0}^{\varepsilon \bk}} \right)[\psi] \right| (\bx) \le C \left\| \psi \right\|_{H^{-1/2}} \varepsilon^2,
	\]
	and that
	\[
		\left| \left( \widetilde{S_{+,e}^{\varepsilon \bk}} - \widetilde{S_{+,e,0}^{\varepsilon \bk}} \right) [\psi] \right| (\bx) + \frac{1}{\varepsilon} \left| \nabla \left(\widetilde{S_{+,e}^{\varepsilon \bk}} - \widetilde{S_{+,e,0}^{\varepsilon \bk}} \right)[\psi] \right| (\bx) \le C \left\| \psi \right\|_{H^{-1/2}} \varepsilon.
	\]
\end{lemma}

Now, we determine the functions $\alpha_1$ and $\alpha_2$ defined in~\eqref{eq:def:alpha0}-\eqref{eq:def:alpha1}.
\begin{lemma} \label{lem:alpha01}
	The function~$- C_{D, \Lat}^+ \widetilde{S_{+}} [\psi_0^+]$ is the solution to~\eqref{eq:def:alpha0}, while the function~$\widetilde{S_{+}} [\psi_1^+ - M_1 C_{D, \Lat}^+ \psi_0^+]$ is the solution to~\eqref{eq:def:alpha1}, \textit{i.e.}
	\[
		- C_{D, \Lat}^+   \widetilde{\cS_+}[\psi_0^+] = \alpha_0  \quad \text{and} \quad
		\widetilde{S_{+}} [\psi_1^+ - M_1 C_{D, \Lat}^+ \psi_0^+] = \alpha_1.
	\]
	In particular, from~\eqref{eq:Se} and the fact that $G_{+,e}^{\bk = \bnull}$ is exponentially decaying, it holds that
	\begin{equation} \label{eq:alpha0infty}
		 \alpha_{0, \infty} = - \frac{1}{| \WS |} \int_{\p D} y_d \left( - C_{D, \Lat}^+ \psi_0^+ \right) = \frac{M_1 C_{D, \Lat}^+}{| \WS |}
	\end{equation}
	and
	\begin{equation} \label{eq:alpha1infty}
		 \alpha_{1, \infty}  = - \frac{1}{| \WS |} \int_{\p D} y_d \left( \psi_1^+ - M_1 C_{D, \Lat}^+ \psi_0^+ \right) = \frac{- \bra y_d, \psi_1^+ \ket}{| \WS |}+ \frac{M_1^2 C_{D, \Lat}^+}{| \WS |}.
	\end{equation}
\end{lemma}

Finally, we are ready to prove our main result Theorem~\ref{th:resonant_metascreen}.

\noindent \textbf{$\bullet$ Proof of Theorem~\ref{th:resonant_metascreen}}

We introduce %(recall that $U_\bk^\rin(\bx) = u_0 \re^{- \ri \bk \cdot \bx}$, so that $\partial_{x_d} U_\bk^\rin(\bar{x}, -x_d) = -\ri k_d u_0 \re^{\ri (k_d x_d - \bar{k}\cdot \bar{x})}$)
 \begin{align*}
 	u_1^\varepsilon(\bx) & :=  (2 \ri u_0 k_d) \widetilde{S_{+,p,0}^{\varepsilon \bk}} \left[ \varepsilon \left( \psi_1^+ - M_1 C_{D, \Lat}^+ \psi_0^+ \right) + g_s(\mu, \varepsilon) \left(C_{D, \Lat}^+ \psi_0^+ \right) \right](\bx) \\
	& = (2 \ri u_0 k_d)  \re^{\ri \varepsilon (k_d x_d - \bar{k} \cdot \bar{x})} \left(\varepsilon \alpha_{1, \infty} - g_s(\mu, \varepsilon) \alpha_{0, \infty} \right)
 \end{align*}
and
\begin{align*}
	u_{\rm BL}^\varepsilon(\bx) & :=  (2 \ri u_0 k_d) \widetilde{S_{+,e,0}^{\varepsilon \bk}} \left[ \varepsilon \left( \psi_1^+ - M_1 C_{D, \Lat}^+ \psi_0^+ \right) +g_s(\mu, \varepsilon) \left(C_{D, \Lat}^+ \psi_0^+ \right) \right](\bx) \\
		& = (2 \ri u_0 k_d)  \re^{- \ri \varepsilon \bar{k} \cdot \bar{x}} \left[  \varepsilon \left( \alpha_1(\bx) - \alpha_{1, \infty} \right) - g_s(\mu, \varepsilon)  \left( \alpha_0 (\bx) - \alpha_{0, \infty} \right)\right].
\end{align*}
By~\eqref{eq:def:us} and Lemma~\ref{lem:GFepsilon}, we deduce that there exists $C \ge 0$ such that,
\[
	\left\| u_\varepsilon^s - \left( u_1^\varepsilon + u_{\rm BL}^\varepsilon \right) \right\|_{L^\infty(S_L)} + 
	\frac{1}{\varepsilon}  \left\| \nabla \left[ u_\varepsilon^s - \left( u_1^\varepsilon +  u_{\rm BL}^\varepsilon \right) \right] \right\|_{L^\infty(S_L)} \le C \left( \varepsilon^2 + \varepsilon | g_s(\mu, \varepsilon) | \right).
\]
Going back to the macroscopic variable $\bX = \varepsilon \bx$ we can conclude the proof of Theorem~\ref{th:resonant_metascreen}.

\section{Numerical illustrations}
\label{sec:numericalIllustrations}
In this section, all the numerical results are obtained for the two-dimensional case. The bubbles are set along a lattice $a \mathbb{Z}$.
We follow the approach taken in~\cite{Ammari2016_Minnaert}, in which the resonant frequencies, for both the one bubble and two bubble cases, were determined numerically. We calculate the characteristic value $\omega_c^+$ of the block operator matrix $\cA$~\eqref{eq:def:Aepsilon} directly. We then use this result to confirm the formula~\eqref{eq:Repsilon} for the periodic bubble resonance:
\[
\omega_M^+ := \sqrt{\frac{\delta C_{D, \Lat}^+}{| D |}} v_b.
\]

\subsection{Implementation details}
Determination of $\omega_c^+$ requires the calculation of the periodic Green function for the Helmholtz equation. It is well known that this function, the solution to~\eqref{eq:Gksharp}, suffers from extremely slow convergence. It can be written in the form
\[
G_\sharp^\bk(\bx, \by) = -\frac{i}{4}\sum_{n \in \Z} e^{i\bar{k}na} H_0^{(1)}(kP_n),
\]
where $H_0^{(1)}$ is the Hankel function of the first kind of order zero, $P_n := \sqrt{(\bar{x} - \bar{y} - na)^2 + (x_d - y_d)^2}$, and $\bar{k} = k \cos(\theta)$ is the component of the wave incident at angle $\theta$ along the $\bar{x}$ direction. The terms of the summation are of the order $O(n^{-\frac{1}{2}})$ for large $n$, which makes the function computationally challenging.

In order to accelerate the convergence, we implement Ewald's method~\cite{ewald, Capolino2005}, tailored to two-dimensional problems featuring one-dimensional periodicity. The periodic Green function is split into two components: a spatial component $G_{\sharp, \rm spat}$ and a spectral component $G_{\sharp, \rm spec}$:
\[
G_\sharp^\bk = G_{\sharp, \rm spat}^\bk + G_{\sharp, \rm spec}^\bk,
\]
where
\begin{align}\label{eq:def:Gspat}
G_{\sharp, \rm spat}^\bk(\bx,\by) &= \frac{1}{4\pi} \sum_{n \in \Z} e^{-i\bar{k}na} \sum_{q=0}^\infty \bigg(\frac{k}{2\cE}\bigg)^{2q}\frac{1}{q!}E_{q+1}(P_n^2\cE^2),
\end{align}
and
\begin{align}\label{eq:def:Gspec}
G_{\sharp, \rm spec}^\bk(\bx,\by) &= \frac{1}{4a} \sum_{p \in \Z} \frac{e^{-ik_{1p}(\bar{x}-\bar{y})}}{ik_{2p}} \times \bigg[e^{ik_{2p}|x_d-y_d|} \text{erfc}\bigg(\frac{ik_{2p}}{2\cE} + |x_d-y_d|\cE \bigg) \\
& + e^{-ik_{2p}|x_d-y_d|} \text{erfc}\bigg(\frac{ik_{2p}}{2\cE} - |x_d-y_d|\cE \bigg) \bigg].
\end{align}
Here, we set
\[
k_{1p} = \bar{k} + \frac{2\pi p}{a}, \quad \mbox{
and } \quad 
k_{2p}= \sqrt{k^2 - k_{1p}^2}.
\]
$k_{1p}$ and $k_{2p}$ are the Floquet wavenumbers along the $\bar{x}$ and $x_d$ directions, respectively. $E_i$ is the exponential integral. $\cE$ is the splitting parameter, which we choose to be $\sqrt{\pi}/d$, an appropriate choice in a low frequency setting, and one which results in exponential convergence in both the $G_{\sharp, \rm spat}$ and $G_{\sharp, \rm spec}$ terms. Due to the exponential convergence of the expressions
 in~\eqref{eq:def:Gspat} and~\eqref{eq:def:Gspec}, we approximate the infinite sums with $n \in [-5, 5]$, $q \in [-15, 15]$, and $p \in [-5, 5]$. We apply corrections to account for periodicity to the usual singular diagonal terms of the discretized operators $\cS_{+}^{\bk}$ and  $\cK_{+}^{\bk, *}$ that comprise the block operator matrix $\cA$.

%An inherent limitation of Ewald's method is that it in order to ensure convergence of a particular contour integral, it is constrained to one half of the complex plane. This means that it is not suitable for use with Muller's method, although it still allows us to inspect the spectrum (on one half of the complex plane).

%The accelerated convergence of the periodic Green's function leads to a drastic reduction in the time it takes to calculate the characteristic value $\omega_c$ of $\cA(\varepsilon)$ using Muller's method. Calculating $\omega_c$ is equivalent to determining the smallest $\omega$ such that $\cA(\varepsilon)$ has a zero eigenvalue, that is,
%\[
%\omega_c = \argmin_{\omega \in \C} \lambda(\omega) = 0, \quad \lambda \in \sigma(\cA(\varepsilon)).
%\]

\subsection{Regime description}
In order to perform the calculations in the appropriate regime, which requires a significant contrast in both the bulk modulii, and the density, of the liquid and the bubbles, we take $\kappa_b$ and $\rho_b$ to be of order $1$, and $\kappa$ and $\rho$ to be of order $10^3$. The wave speeds inside and outside the bubbles are both of order $1$.

\subsection{Validation of the periodic resonant frequency formula}
We begin by calculating the resonant frequency given by the formula. In order to obtain the capacity $C_{D, \Lat}^+$, we first compute the eigenfunction $\psi$ corresponding to the eigenvalue $\frac{1}{2}$ for the operator $\cK_{+}^{*} = \cK_{+}^{0,*}$. We approximate a basis for the discretized version of this operator (and also for the single layer potential) with the family of functions having a value of $1$ at a particular point, and $0$ everywhere else. We fix the period to be $a = 10$, and take a set of bubble radii in the range $r \in [0.1, 1]$. The resonant frequencies obtained with the formula are given in Table~\ref{tab:periodic-resonance}. The characteristic values of $\cA$ are shown in Figure~\ref{im:periodic-resonance-dirichlet}. It is clear that the characteristic values $\omega_c^+$ correspond to the resonant frequencies $\omega_M^+$ obtained with the formula,  confirming its validity.

\setlength{\tabcolsep}{12pt}
\begin{table} [!t]
    \centering
    \begin{tabular}{cc}
        %\toprule
        $r$ & $\omega_M^+$ \\        
       \medskip
        0.1000 & 0.3898 \\
        0.3250 & 0.1191 \\
        0.5500 & 0.0694 \\
        0.7750 & 0.0483 \\
        1.0000 & 0.0366 \\
        %\bottomrule
    \end{tabular}
    \caption{The resonant frequencies $\omega_M^+$ obtained with the formula for a set of radii in the interval $[0.1, 1]$. The spectrum of the corresponding block operator matrix $\cA$ from Equation~\eqref{eq:def:Aepsilon}, which was calculated explicitly, is shown in Figure~\ref{im:periodic-resonance-dirichlet}.}\label{tab:periodic-resonance}
\end{table}

\begin{figure} [!t]
\begin{center}
  \includegraphics[width=1\textwidth]{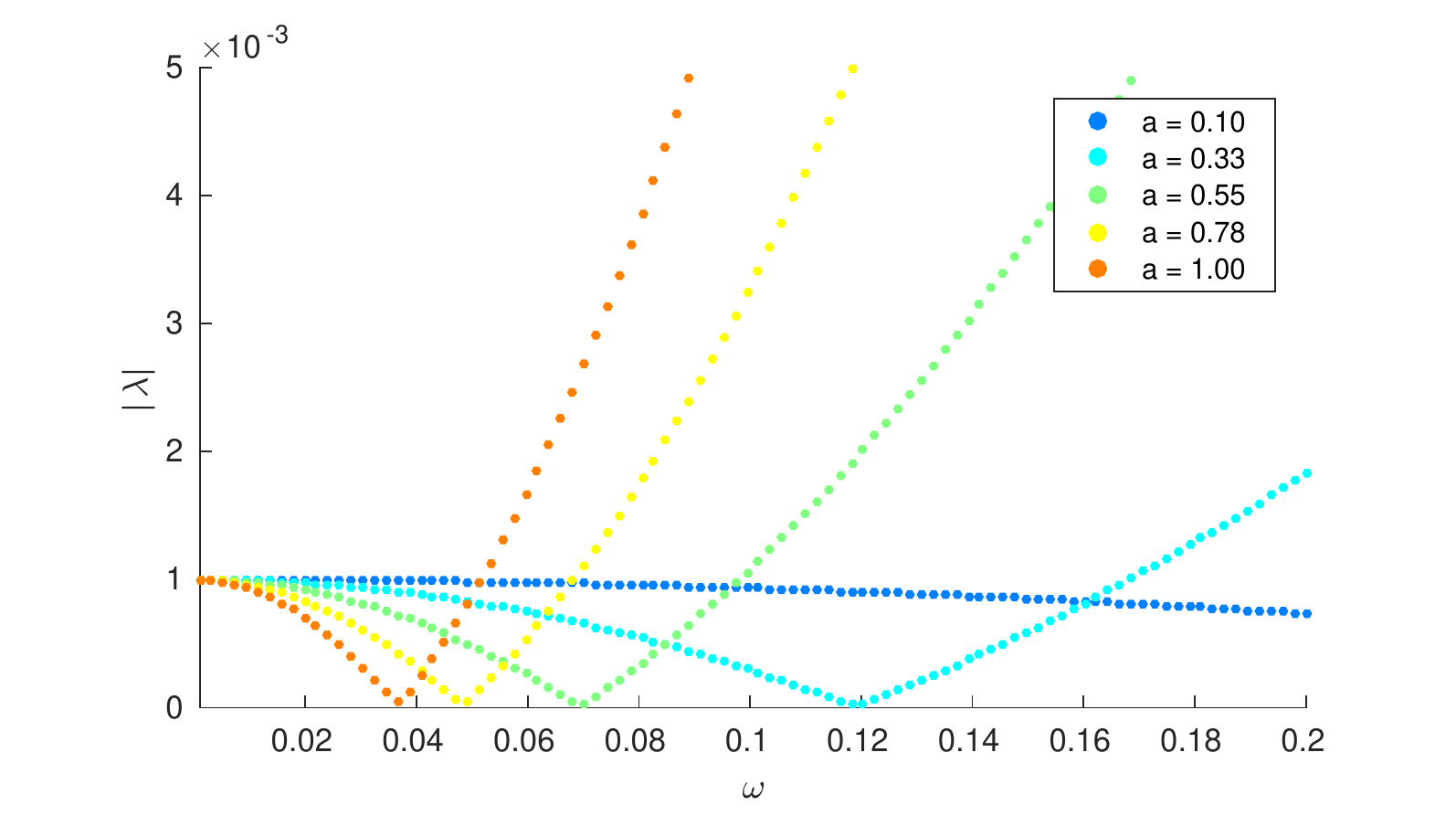}
  \caption{The smallest eigenvalue in the spectrum of the block operator matrix $\cA$ from Equation~\eqref{eq:def:Aepsilon}, for a period of $a = 10$, and a set of bubble radii in the interval $[0.1, 1]$. We have characteristic values $\omega_c^+$ where the eigenvalues go to $0$, indicating resonance. The characteristic values correspond to the resonant frequencies $\omega_M^+$, given in Table~\ref{tab:periodic-resonance}, obtained with the formula for the same parameters.}
  \label{im:periodic-resonance-dirichlet}
\end{center}
\end{figure}

\subsection{Effect of periodicity and bubble radii on resonance}
Let $\beta$ be the distance from the bubble centers to the reflective plane $\partial D$. In Figure~\ref{im:periodic-resonance-over-periods}, we fix the bubble radii at $r = 1$, and analyze the relationship between periodicity and the resonant frequency, as we increase $\beta$, moving the bubbles further from the reflective plane. We find that $\omega_M^+$ has a logarithmic dependency on the period, and is inversely proportional to the distance from the plane.

\medskip

Similarly in Figure~\ref{im:periodic-resonance-over-radii}, where this time we fix the period at $a = 5$, and consider the relationship between the bubble radii and the resonant frequency as the distance from the plane varies. Although resonant frequencies of bubbles are known to be inversely proportional to their radii, here we find that when the bubble radii are increased such that the bubbles are almost touching the reflective plane, the resonant frequency $\omega_M^+$ in fact increases as we further increase the radii.

\begin{figure} [!t]
\begin{center}
  \includegraphics[width=1\textwidth]{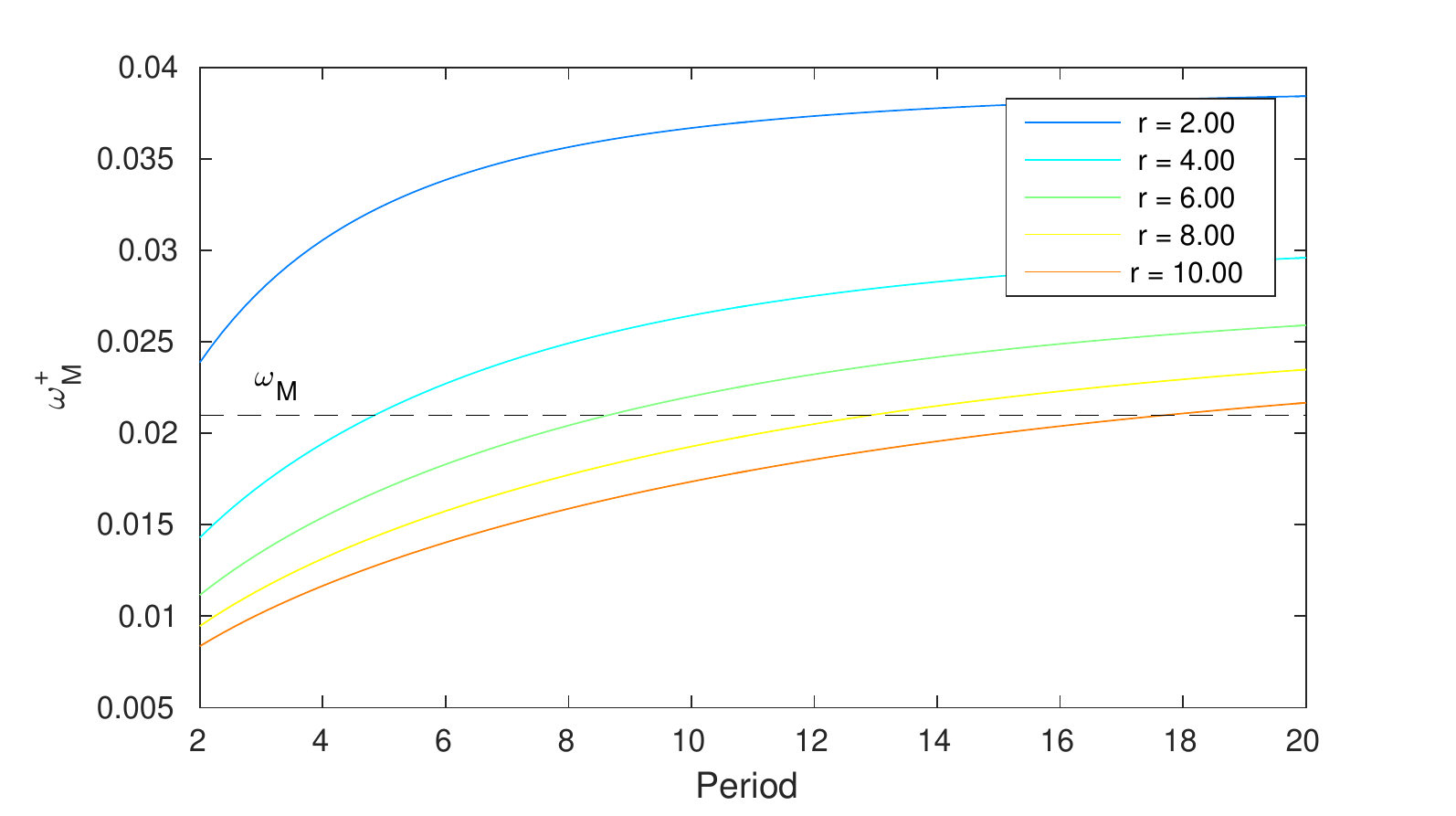}
  \caption{The bubble radii are fixed at $r = 1$, and the distance from the bubble centers to the reflective plane $\partial D$ is represented by $\beta$. $\omega_M^+$ has a logarithmic dependency on the period $a$, and decreases as the distance $r$ from the reflective plane $\partial D$ is increased. The one bubble resonant frequency $\omega_M$, for a bubble of radius $1$ in free space, is plotted for reference.}
  \label{im:periodic-resonance-over-periods}
\end{center}
\end{figure}

\begin{figure} [!t]
\begin{center}
  \includegraphics[width=1\textwidth]{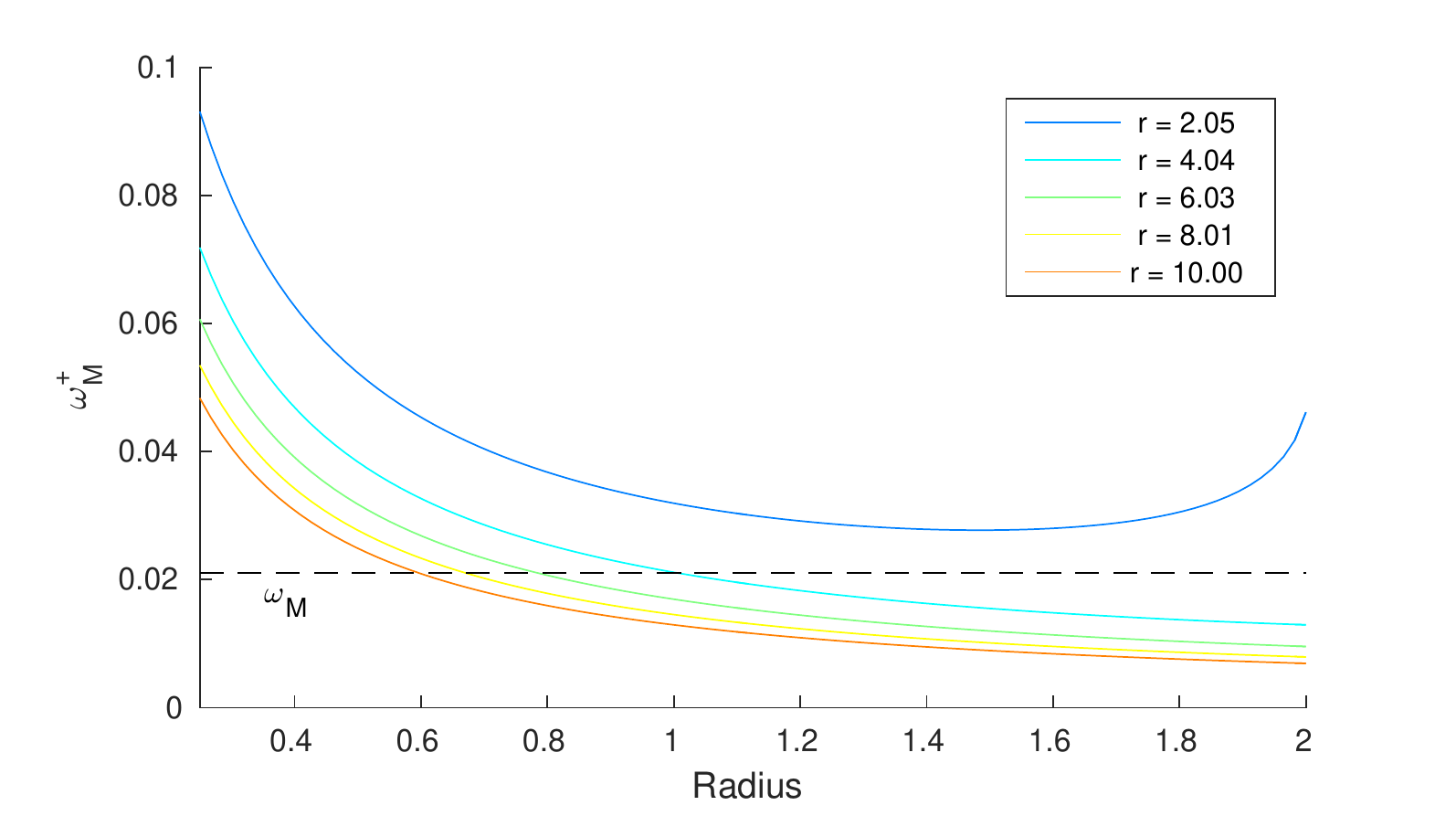}
  \caption{The period is fixed at $a = 5$, and the distance from the bubble centers to the reflective plane $\partial D$ is represented by $\beta$. The resonant frequency of a bubble is inversely proportional to its radius. However, in the case of a periodic lattice of bubble near a reflective plane, when the bubble radii $r$ become large enough such that the bubbles are almost touching the plane, the resonant frequency $\omega_M^+$ in fact increases as we further increase the radii. Again, the one bubble resonant frequency $\omega_M$, for a bubble of radius $1$ in free space, is plotted for reference.}
  \label{im:periodic-resonance-over-radii}
\end{center}
\end{figure}

%%%%%%%%%%%%%%%%%%%%%%%%%%%%%%%%%%

%----------------------- BIBLIO ---------------------------
%%%%%%%%%%%%%%%%%%%%%%%%%%%%%%%%%%

%\newpage
\bibliographystyle{plain}

%\bibliography{metascreen}

\end{document}